\documentclass[11pt, reqno]{amsart}
\usepackage{amssymb,mathrsfs,graphicx,extpfeil}
\usepackage{epsfig}
\usepackage{indentfirst, latexsym, amssymb, enumerate,amsmath,graphicx}
\usepackage{float}
\usepackage{colortbl}
\usepackage{epsfig,subfigure}
\usepackage{mathtools}
\usepackage{amsmath}
\title[The Kuramoto model with inertia and frustration]{Emergent dynamics of the inertial Kuramoto model with frustration on a locally coupled graph}

\author[Zhu]{Tingting Zhu \textsuperscript{\MakeLowercase{a,b}}}

\author[Zhang]{Xiongtao Zhang \textsuperscript{\MakeLowercase{c},*}}

\newtheorem{theorem}{Theorem}[section]
\newtheorem{lemma}{Lemma}[section]

\newtheorem{remark}{Remark}[section]

\newtheorem{definition}{Definition}[section]

\def\charf {\mbox{{\text 1}\kern-.30em {\text l}}}
\begin{document}
%%%%%%%%%%%%%%%%

\date{\today}


\thanks{\textsuperscript{a} School of Artificial Intelligence and Big Data, Hefei University, Hefei, China (ttzhud201880016@163.com)}
\thanks{\textsuperscript{b} Key Laboratory of Applied Mathematics and Artificial Intelligence Mechanism, Hefei University, Hefei, China}
\thanks{\textsuperscript{c} School of Mathematics and Statistics, Wuhan University, Wuhan, China (zhangxt@whu.edu.cn)}
\thanks{\textsuperscript{*} Corresponding author. }

\begin{abstract}
We study the synchronized behavior of the inertial Kuramoto oscillators with frustration effect under a symmetric and connected network. Due to the lack of second-order gradient flow structure and singularity of second-order derivative of diameter, we shift to construct convex combinations of oscillators and related new energy functions that can control the phase and frequency diameters. Under sufficient frameworks on initial data and system parameters, we derive first-order dissipative differential inequalities of constructed energy functions. This eventually gives rise to the emergence of frequency synchronization exponentially fast.
\end{abstract}
\maketitle \centerline{\date}

%\tableofcontents
\section{Introduction}\label{sec:1}
\vspace{0.5cm}

Synchronization phenomena of finite population of oscillatory systems has raised research interest in various scientific communities such as biology, engineering control, physics \cite{A-B-P-R05,B-B66,Do-B12,P-R-K01,S00,W67} and etc. The Kuramoto model \cite{K75} is one of the well-known models that effectively depict the collective synchronized behaviors observed in these systems. Due to the potential applications, theoretical studies for the first order Kuramoto model have been extensively investigated from different aspects, to name a few, synchronization analysis \cite{C-H-J-K12,H-K-P15, H-R20}, network structure \cite{D-X13, Z-Z23}, mean-filed limit \cite{B-C-M15}, physical effects \cite{H-K-L14,S-K86} and so on. To better describe some synchronous patterns, the second-order Kuramoto type model was proposed \cite{E91} with the consideration of inertial effects. This model has been extensively explored for a deeper grasp of superconducting Josephson junctions \cite{W-S97,W-C-S96}  and power grids \cite{D-B12,L-X-Y14}. In this work, we focus our attention on the inertial Kuramoto model under the joint effects of phase shifts (also called frustration ) and network structure.

To set up the stage, consider an interaction network characterized by a weighted graph $G = (V,E,\Psi)$. This graph consists of a vertex set $V = \{1, 2, \cdots, N\}$, an edge set $E \subset V \times V$, and an $N \times N$ matrix $\Psi = (\psi_{ij})$ whose element $\psi_{ij}$ represents the capacity of the communication weight flowing from vertex $j$ to $i$. Note that the edge set can be further denoted as $E = \{(j,i) | \psi_{ij} > 0\}$. Assume that the Kuramoto oscillators are located at the vertices of the graph $G$, and interactions between them are registered by $E$ and $\Psi$. Let $\theta_i = \theta_i(t) \in \mathbb{R}$ be the phase of the $i$-th oscillator, and the phase dynamics of the inertial Kuramoto oscillators under the effect of frustration is governed by the following second-order system:
\begin{equation}\label{s_phs}
m \ddot{\theta}_i(t) + \dot{\theta}_i(t) = \Omega_i + \frac{K}{N} \sum_{j=1}^N \psi_{ij} \sin (\theta_j(t) - \theta_i(t) + \alpha), \quad t > 0, \ i=1,2, \ldots, N,
\end{equation}
subject to the initial data:
\begin{equation}\label{initial_data}
(\theta_i(0), \omega_i(0)) = (\theta_{i0}, \omega_{i0}).
\end{equation}
Here, $\Omega_i$ is the natural frequency and $\omega_i(t) = \dot{\theta}_i(t)$ is the instantaneous frequency of the $i$-th oscillator, $m > 0$ and $0 < \alpha < \frac{\pi}{2}$ respectively denote the strength of inertia and frustration effects, and $K > 0$  reprensents the coupling strength. 
%We assume that the interaction matrix $\Psi = (\psi_{ij})$ is symmetric and connected:
%\begin{enumerate}[(i)]
%\item $\psi_{ij} = \psi_{ji} \ge 0, \ 1 \le i \ne j \le N$ and $\psi_{ii} = 0, \ 1 \le i \le N$.\\
%
%\item For any $i,j \in V$, there exists one path from $i$ to $j$, i.e.,
%\begin{equation*}
%i  = i_0 \rightarrow i_1 \rightarrow \cdots \rightarrow i_p =j, \quad (i_{l-1}, i_l) \in E, \ 1 \le l \le p.
%\end{equation*}
%\end{enumerate}

%{\color{red}(Review previous works...)}\newline
For the all-to-all coupling case, i.e., $\psi_{ij} = 1$, the second-order phase model \eqref{s_phs} with zero frustration ($\alpha = 0$) has been studied in \cite{C-H-Y11} and complete synchronization estimates for some restricted class of initial configurations were provided. Moreover, the formational structure of phase-locked state was analyzed in \cite{C-H-N13}. Later on, the authors in \cite{C-H-M18,C-L-H-X-Y14,H-J-K19} extended the results to more generic initial data. On the other hand, for the locally coupled case, the authors in \cite{C-L19,C-L-H-X-Y14,L-X-Y14,W-Q17} addressed the synchronization problem of system \eqref{s_phs} without frustration on a symmetric and connected network, and sufficient conditions on initial phase distributions and system parameters leading to complete synchronization were presented. Note that there is no information on the convergence rate in \cite{C-L19,C-L-H-X-Y14,W-Q17}. Additionally, the interplay between inertia and frustration on an all-to-all network were investigated in \cite{Ha-K-L14, L-H16}, and the occurrence of asymptotic synchronized behavior and well-ordering property were shown. For more information, we refer to the literatures \cite{D-Z-P-L-J18,H-K-P18,H-K-Z18, H-C-Y-S99,T-L-O97,Ta-L-O97,W-C21}. 
%The authors in \cite{Z-Z23} have considered the locally coupled Kuramoto model with inertia and frustration, and presented sufficient frameworks to ensure the emergence of complete synchronization. However, the initial phase diameter is relatively restricted and required to be small may due to the usage of $l_2$ estimate.

In this paper, we will study the inertial Kuramoto model with frustration on a symmetric and connected network. As far as we know, there is comparatively less theoretical work on the joint effects of inertia, frustration, and network structure. It is natural to ask the following questions:
\begin{enumerate}
\item Will the frequency synchronization asymptotically occur for the inertial Kuramoto model with frustration \eqref{s_phs} on a locally coupled graph?

\item How fast is the rate of convergence to synchronized behavior?
\end{enumerate}
%{1. emergence of synchronization for inertial Kuramoto model with inertia and frustration on a digraph
%
%2. how fast is the emergence rate.\newline}
There are mainly three difficulties to answer the above two questions. Below, we will briefly discuss these difficulties and introduce our corresponding strategy. For convenience, we will list all the notations in the end of this section. 

Firstly, due to the inertial term, the oscillation behaviors are not negligible and it is more difficult to make the phase diameter uniformly bounded. Therefore, we need stronger constraints on the initial configurations and system parameters. More precisely, we assume:\newline

\noindent \textit{{\bf Assumption $(\mathcal{H}1)$:}  The initial configurations satisfy the following condition,
\begin{equation}2m D(\Omega) + 4mK\psi_u \sin \alpha + (1+4mK\psi_u \cos \alpha) D_\theta(0) + 3m D_\omega(0) < \beta,\label{assum_2}\end{equation}}
where $\beta \in (0,\pi)$ is a suitable constant. 

$\ $

\noindent The above assumption shows that, the initial phase diameter $D_\theta(0)$ is allowed to be close to $\pi$ when the strength of inertia $m$ is small. The authors in \cite{Z-Z23} have considered the locally coupled Kuramoto model with inertia and frustration, and presented sufficient frameworks to ensure the emergence of complete synchronization. However, the initial phase diameter is relatively restricted and required to be small even for small inertia may due to the usage of $l_2$ estimate approach. On the other hand, when the strength of inertia $m$ is large, the oscillations will be strong, then the initial phase diameter is necessary to be small in above assumption. No matter the cases, the assumption allows us to prove that the oscillators will keep staying in a half circle.  

Secondly, the connection topology considered in this paper may not be an all-to-all network. Therefore, the dynamics of the diameters lack the uniform damping. Moreover, the frustration term in sine coupling function results in the lack of second-order gradient-like flow structure, hence, the energy method employed in \cite{C-L19,C-L-H-X-Y14} can not be directly applied in our setting.  To handle this difficulty and inspired by the approach employed in \cite{H-L-Z20}, we intend to study the hypo-coercive dynamics of the convex combination of the oscillators (see \eqref{Q_function}, \eqref{F_function}, \eqref{A_function} and \eqref{B_function}) for the emergence of synchrony. Owing to the symmetric property of network, we take more simpler constructions of convex combinations to address the problem.
%$(\sum_{n=0}^{N-1} c^n)^{-1} \sum_{i=1}^N c^{i-1} z_{l_i}(t)$, where $(l_i)$ is some permutation of $(i)$.
We require the following assumptions on the coefficient $c$ of  convex combination:\newline

\noindent \textit{{\bf Assumption $(\mathcal{H}2)$:} assume the positive constants $\gamma$, $D^\infty$, $\delta$ and $c$ satisfy the following constraints,
\begin{equation}
\beta<\gamma<\pi,\quad 0 < D^\infty <\min \left\{\frac{\pi}{2}, \beta \right\},\quad \alpha < \delta < \frac{\pi}{2}-D^\infty,\quad  c> \max\left\{\frac{2}{\cos (D^\infty + \delta)}, \frac{2}{1 - \frac{\beta}{\gamma}}\right\}.\label{assume_1}
\end{equation}}

$\ $

\noindent On one hand, we can prove the dynamics of the convex combinations are dissipative. On the other hand, by choosing a proper coefficient $c$ showing in \eqref{assume_1}, we may prove the diameters of the convex combinations (maximal one minus minimal one) are correspondingly equivalent to the diameters of phase, frequency, acceleration and jerk (see \eqref{phs_size}, \eqref{fre_size}, \eqref{acce_size} and \eqref{B_size}). Then, the dissipation of the convex combinations can eventually imply the hypo-coercivity of the phase and frequency diameters. 

Thirdly, 
%In the all-to-all case, one may show well ordering structure such as in \cite{L-H16}, so that the diameter $D(\theta(t))=\theta_i-\theta_j$ for fixed $i,j$. This allows the diameter to be second-order smooth and its dynamics can be effectively described by a second-order equation as shown in \cite{Ha-K-L14}. However, establishing such a well-ordering structure for the model \eqref{s_phs} on a network still remains an open problem. 
the diameters of the convex combinations are only Lipschitz continuous and their second order derivatives contains singularities. Therefore, even we get the dissipative structures in the second order differential equations, it is not sufficient to conclude the emergence of synchronization. To address this issue, we combine the diameters of convex combinations together to construct new energy functions $\mathcal{E}_1$ and $\mathcal{E}_2$ (see \eqref{energy_phs} and \eqref{energy_fre} for details of the construction), and propose the following assumptions on system parameters:\newline

\noindent\textit{{\bf Assumption $(\mathcal{H}3)$:} Let $\eta=1-2c^{-1}$ and assume the coupling strength $K$ and the inertia $m$ satisfy the following constraints,
\begin{align}
&mK < \min \left\{\frac{\eta^2\psi_l \sin \gamma }{16N\psi_u^2(\sum_{n=0}^{N-1}c^n) \gamma}, \frac{N (\sum_{n=0}^{N-1} c^n)}{2\psi_l } \right\},
%\quad mK < \frac{N(\sum_{n=0}^{N-1} c^n)}{\psi_l}
% \quad mK < \frac{\eta^2 \psi_l}{16N\psi_u^2 (\sum_{n=0}^{N-1} c^n)}
 \label{assume_3}\\
&K >  \frac{4 (D(\Omega) + 2K \psi_u \sin \alpha)N(\sum_{n=0}^{N-1} c^n)\gamma}{D^\infty\eta \psi _l \cos \alpha \sin \gamma} ,
%&\frac{4 (D(\Omega) + 2K \psi_u \sin \alpha)N(\sum_{n=0}^{N-1} c^n)\gamma}{K\psi _l \cos \alpha \sin \gamma} < \beta, \quad 
%&\frac{4 (D(\Omega) + 2K \psi_u \sin \alpha)N(\sum_{n=0}^{N-1} c^n)\gamma}{\eta K\psi _l \cos \alpha \sin \gamma} \le D^\infty, 
\label{assume_4}\\
&K > \frac{64 (D(\Omega) + 2K \psi_u \sin \alpha)(D^\infty\cos \alpha + \sin \alpha)N^3\psi^2_u(\sum_{n=0}^{N-1} c^n)^3\gamma^2 }{ \eta^3 \psi _l^3 \cos \alpha \sin^2 \gamma}.\label{assume_5}
\end{align}}

$\ $

\noindent It is easy to check that, when $K$ is sufficiently large, and $\alpha$ and $m$ are sufficiently small, the set of parameters satisfying above assumptions is not empty. Under above assumptions, we can prove the dynamics of the energy functions are governed by first order dissipative differential inequalities (see Lemma \ref{phs_sum_dynamics} and Lemma \ref{syn}). Then, the Lipschitz continuity is enough for us to obtain the synchronization estimates (see Theorem \ref{main}).   

Now, we are ready to claim our main result.

\begin{theorem}\label{main}
Let $(\theta(t), \omega(t))$ be a solution to system \eqref{s_phs} and suppose Assumption $(\mathcal{H}1)-(\mathcal{H}3)$ hold. Then, there exists a time $t_* \ge 0$ such that
\begin{equation*}
D_\omega(t) \le Ce^{- \Lambda (t-t_*)}, \quad t \ge t_*,
\end{equation*}
where $C$ and $\Lambda$ are positive constants depending on initial data and system parameters.
\end{theorem}

The rest of the paper are organized as below. In Section \ref{sec:2}, we introduce the construction of convex combinations of phase, frequency, acceleration and jerk, and show the equivalence relation between the diameter of convex combination and the diameter of original state quantity. In Section \ref{sec:3}, we present a first-order differential inequality of energy function $\mathcal{E}_1$ that can control the phase diameter, and show that the phase diameter will be uniformly bounded by a small value after some finite time. In Section 4, we derive a dissipative differential inequality of energy function $\mathcal{E}_2$ that can dominate the frequency diameter. This ultimately yields the occurrence of frequency synchronization exponentially fast. Section \ref{sec:5} is devoted to a brief summary of the main result.
\newline

\noindent {\bf Notation:} For notational simplicity, we introduce some notations which will be frequently used throughout the paper:
\begin{equation*}
\begin{aligned}
& \theta(t) = (\theta_1(t), \cdots, \theta_N(t)), \quad D_\theta(t) = \max_{1\le i \le N} \theta_i(t) - \min_{1 \le i \le N} \theta_i(t),\\
&\omega_i(t) = \dot{\theta}_i(t), \quad \omega(t) = (\omega_1(t), \cdots, \omega_N(t)), \quad D_\omega(t) = \max_{1 \le i \le N} \omega_i(t) - \min_{1 \le i \le N} \omega_i(t),\\
&\Omega_M =\max_{1 \le i \le N} \Omega_i, \quad \Omega_m = \min_{1 \le i \le N} \Omega_i,\quad D(\Omega) = \Omega_M - \Omega_m , \quad \psi_u = \max_{(i,j) \in E} \psi_{ji}, \quad \psi_l = \min_{(i,j) \in E} \psi_{ji},\\
&a_i(t) = \dot{\omega}_i(t), \quad a(t) = (a_1(t),\cdots,a_N(t)), \quad D_a(t) = \max_{1 \le i,j \le N} |a_i(t) - a_j(t)|,\\
&b_i(t) = \dot{a}_i(t), \quad b(t) = (b_1(t),\cdots,b_N(t)), \quad D_b(t) = \max_{1 \le i,j \le N} |b_i(t) - b_j(t)|.\\
\end{aligned}
\end{equation*}

%\noindent {\bf Assumption $(\mathcal{H})$:} Our assumptions on system parameters are given as follows:
%\begin{align}
%& 0 < \beta < \gamma < \pi, \quad c > \max\left\{\frac{2}{\cos (D^\infty + \delta)}, \frac{2}{1 - \frac{\beta}{\gamma}}\right\}, \quad \eta = 1 - \frac{2}{c},  \quad 0 < D^\infty <\min \left\{\frac{\pi}{2}, \beta \right\}, \quad D^\infty + \delta < \frac{\pi}{2}, \label{assume_1}\\
%&\alpha < \delta, \quad 2m D(\Omega) + 4mK\psi_u \sin \alpha + (1+4mK\psi_u \cos \alpha) D_\theta(0) + 3m D_\omega(0) < \beta,\label{assume_2}\\
%&mK < \min \left\{\frac{\eta^2\psi_l \sin \gamma }{16N\psi_u^2(\sum_{n=0}^{N-1}c^n) \gamma}, \frac{N (\sum_{n=0}^{N-1} c^n)}{2\psi_l } \right\},
%%\quad mK < \frac{N(\sum_{n=0}^{N-1} c^n)}{\psi_l}
%% \quad mK < \frac{\eta^2 \psi_l}{16N\psi_u^2 (\sum_{n=0}^{N-1} c^n)}
% \label{assume_3}\\
%&K >  \frac{4 (D(\Omega) + 2K \psi_u \sin \alpha)N(\sum_{n=0}^{N-1} c^n)\gamma}{D^\infty\eta \psi _l \cos \alpha \sin \gamma} ,
%%&\frac{4 (D(\Omega) + 2K \psi_u \sin \alpha)N(\sum_{n=0}^{N-1} c^n)\gamma}{K\psi _l \cos \alpha \sin \gamma} < \beta, \quad 
%%&\frac{4 (D(\Omega) + 2K \psi_u \sin \alpha)N(\sum_{n=0}^{N-1} c^n)\gamma}{\eta K\psi _l \cos \alpha \sin \gamma} \le D^\infty, 
%\label{assume_4}\\
%&K > \frac{64 (D(\Omega) + 2K \psi_u \sin \alpha)(D^\infty\cos \alpha + \sin \alpha)N^3\psi^2_u(\sum_{n=0}^{N-1} c^n)^3\gamma^2 }{ \eta^3 \psi _l^3 \cos \alpha \sin^2 \gamma},\label{assume_5}\\
%\end{align}
%where $c, \beta, \gamma, \eta, D^\infty, \delta$ are suitable positive constants.\newline

\section{Preliminaries}\label{sec:2}
\setcounter{equation}{0}
In this section, we review some elementary concepts and estimates,  and introduce the constructions of convex combinations and new energy functions.

First, we recall the definition of complete synchronization.

\begin{definition}
Let $(\theta(t), \omega(t))$ be a solution to system \eqref{s_phs}. The system exhibits asymptotic complete frequency synchronization if and only if the relative frequency difference tends to zero asymptotically:
\begin{equation*}
\lim_{t \to \infty} (\omega_i(t) - \omega_j(t)) = 0, \quad \text{for all} \ i,j = 1,2,\cdots, N.
\end{equation*}
\end{definition}

Next, we present the construction of convex combination.
 Let $z(t) = \{z_i(t)\}_{i=1}^N$ be the state quantity of an ensemble of N-oscillators and $c>0$ is a suitable constant mentioned in \eqref{assume_1}. For any given time $t$, we order the oscillators' state quantity $z(t)$ from minimum to maximum
\begin{equation}\label{random_order}
z_{l_1}(t) \le z_{l_2}(t) \le \cdots \le z_{l_N}(t),
\end{equation}
where $l_1 l_2 \cdots l_N$ is a permutation of $\{1,2,\cdots,N\}$ depending on the state $z(t)$ at time $t$. Then, we construct a new function $Y(t)$ related to convex combinations of $z(t)$:
\begin{equation}\label{Y_function}
Y(t) = \bar{z}(t) - \underline{z}(t),
\end{equation}
where
\begin{align}
\bar{z}(t) = \frac{1}{\sum_{n=0}^{N-1} c^n} \sum_{i=1}^N c^{i-1} z_{l_i}(t), \quad \underline{z}(t) = \frac{1}{\sum_{n=0}^{N-1} c^n} \sum_{i=1}^{N} c^{N-i} z_{l_i}(t). \label{barunder_z}
\end{align}
The following lemma shows that the function $Y(t)$ is equivalent to the diameter $D_z(t)$. 

\begin{lemma}\label{size}
Let $z(t) = \{z_i(t)\}_{i=1}^N$ be the state quantity of an ensemble of N-oscillators at time $t$.
Then, we have
\begin{equation*}
\eta D_z(t)\le Y(t) \le D_z(t), \quad \eta = 1- \frac{2}{c},
\end{equation*}
where $Y(t)$ is defined in \eqref{Y_function} and $D_z(t) = \max\limits_{1 \le i \le N} z_i(t) - \min\limits_{1 \le i \le N} z_i(t)$.
\end{lemma}
\begin{proof}
For the sake of discussion, at time $t$, we assume without loss of generality
\begin{equation}\label{regular_order}
z_1(t) \le z_2 (t) \le \cdots \le z_N(t),
\end{equation}
if necessary, we still adopt \eqref{random_order}.
Then based on \eqref{regular_order} and the construction of $Y(t)$ in \eqref{Y_function}, we have
\begin{equation}\label{Y_definition}
Y(t) = \bar{z}(t) - \underline{z}(t), \quad \bar{z}(t) = \frac{1}{\sum_{n=0}^{N-1} c^n} \sum_{i=1}^N c^{i-1} z_{i}(t), \quad \underline{z}(t) = \frac{1}{\sum_{n=0}^{N-1} c^n} \sum_{i=1}^{N} c^{N-i} z_{i}(t).
\end{equation}
It is easy to see that
\begin{equation*}
Y(t) \le z_N(t) - z_1(t) = D_z(t).
\end{equation*}
Next, it follows from \eqref{Y_definition} that
\begin{equation*}
\begin{aligned}
Y(t) &= \bar{z}(t) - \underline{z}(t) = \frac{1}{\sum_{n=0}^{N-1} c^n} \sum_{i=1}^N c^{i-1} z_{i}(t) -  \frac{1}{\sum_{n=0}^{N-1} c^n} \sum_{i=1}^{N} c^{N-i} z_{i}(t)\\
%&= z_N(t) - z_1(t) + \frac{1}{\sum_{n=0}^{N-1} c^n} \sum_{i=1}^N c^{i-1} z_{i}(t) - z_N(t) + z_1(t) - \frac{1}{\sum_{n=0}^{N-1} c^n} \sum_{i=1}^{N} c^{N-i} z_{i}(t)\\
&= z_N(t) - z_1(t) + \frac{1}{\sum_{n=0}^{N-1} c^n} \sum_{i=1}^{N-1} c^{i-1} (z_{i}(t)-z_N(t)) + \frac{1}{\sum_{n=0}^{N-1} c^n} \sum_{i=2}^{N} c^{N-i} (z_1(t) -z_{i}(t)).
\end{aligned}
\end{equation*}
Then, according to \eqref{regular_order}, we have
\begin{equation*}
\begin{aligned}
Y(t)&\ge z_N(t) - z_1(t) + \frac{1}{\sum_{n=0}^{N-1} c^n} \sum_{i=1}^{N-1} c^{i-1} (z_{1}(t)-z_N(t)) + \frac{1}{\sum_{n=0}^{N-1} c^n} \sum_{i=2}^{N} c^{N-i} (z_1(t) -z_{N}(t))\\
%&= \left( 1 - \frac{1}{\sum_{n=0}^{N-1} c^n} \sum_{i=1}^{N-1} c^{i-1} -  \frac{1}{\sum_{n=0}^{N-1} c^n} \sum_{i=2}^{N} c^{N-i} \right) (z_N(t) - z_1(t)) \\
&= \left( 1 - \frac{2}{\sum_{n=0}^{N-1} c^n} \sum_{i=1}^{N-1} c^{i-1} \right) (z_N(t) - z_1(t)) \ge \left(1 - \frac{2}{c}\right) D_z(t),
\end{aligned}
\end{equation*}
where we used the fact
\begin{equation*}
\frac{2}{\sum_{n=0}^{N-1} c^n} \sum_{i=1}^{N-1} c^{i-1} = 2\frac{c^{N-1} - 1}{c^N -1} \le 2\frac{c^{N-1} - 1}{c^N -c}  = \frac{2}{c}.
\end{equation*}
%Therefore, we finish the proof of this lemma.
\end{proof}

In the sequel, we apply the construction principle in \eqref{Y_function} and Lemma \ref{size} to construct crucial functions as follows:

\begin{enumerate}
\item {\bf (Definition of $Q(t)$)} For any given time $t \ge 0$, define 
\begin{equation}\label{Q_function}
Q(t) = \bar{\theta}(t) - \underline{\theta}(t),
\end{equation}
where
\begin{align*}
\bar{\theta}(t) = \frac{1}{\sum_{n=0}^{N-1} c^n} \sum_{i=1}^N c^{i-1} \theta_{l_i}(t), \quad \underline{\theta}(t) = \frac{1}{\sum_{n=0}^{N-1} c^n} \sum_{i=1}^{N} c^{N-i} \theta_{l_i}(t), \label{barunder_Q}
\end{align*}
and $l_1l_2\cdots l_N$ is a permutation of $\{1,2,\cdots, N\}$ depending on the phase value $\theta(t)$ at time $t$ such that
\begin{equation}\label{random_order1}
\theta_{l_1}(t) \le \theta_{l_2}(t) \le \cdots \le \theta_{l_N}(t).
\end{equation}
Moreover, we have
\begin{equation}\label{phs_size}
\begin{aligned}
\eta D_\theta(t) \le Q(t) \le D_\theta(t).
\end{aligned}
\end{equation}

\item {\bf (Definition of $F(t)$)} For any given time $t \ge 0$, define
\begin{equation}\label{F_function}
F(t) = \bar{\omega}(t) - \underline{\omega}(t),
\end{equation}
where
\begin{align*}
\bar{\omega}(t) = \frac{1}{\sum_{n=0}^{N-1} c^n} \sum_{i=1}^N c^{i-1} \omega_{k_i}(t), \quad \underline{\omega}(t) = \frac{1}{\sum_{n=0}^{N-1} c^n} \sum_{i=1}^{N} c^{N-i} \omega_{k_i}(t), \label{barunder_F}
\end{align*}
and $k_1k_2\cdots k_N$ is a permutation of $\{1,2,\cdots, N\}$ depending on the frequency value $\omega(t)$ at time $t$ such that
\begin{equation}\label{random_order3}
\omega_{k_1}(t) \le \omega_{k_2}(t) \le \cdots \le \omega_{k_N}(t).
\end{equation}
Moreover, we have
\begin{equation}\label{fre_size}
\eta D_\omega(t) \le F(t) \le D_\omega(t).
\end{equation}

\item {\bf (Definition of $A(t)$)} For any given time $t \ge 0$, define
\begin{equation}\label{A_function}
A(t) = \bar{a}(t) - \underline{a}(t),
\end{equation}
where
\begin{align*}
\bar{a}(t) = \frac{1}{\sum_{n=0}^{N-1} c^n} \sum_{i=1}^N c^{i-1} a_{p_i}(t), \quad \underline{a}(t) = \frac{1}{\sum_{n=0}^{N-1} c^n} \sum_{i=1}^{N} c^{N-i} a_{p_i}(t), \label{barunder_F}
\end{align*}
and $p_1p_2\cdots p_N$ is a permutation of $\{1,2,\cdots, N\}$ depending on the acceleration value $a(t)$ at time $t$ such that
\begin{equation}\label{random_order4}
a_{p_1}(t) \le a_{p_2}(t) \le \cdots \le a_{p_N}(t).
\end{equation}
Moreover, we have
\begin{equation}\label{acce_size}
\eta D_a(t) \le A(t) \le D_a(t).
\end{equation}

\item {\bf (Definition of $B(t)$)} For any given time $t \ge 0$, define
\begin{equation}\label{B_function}
B(t) = \bar{b}(t) - \underline{b}(t),
\end{equation}
where
\begin{align*}
\bar{b}(t) = \frac{1}{\sum_{n=0}^{N-1} c^n} \sum_{i=1}^N c^{i-1} b_{q_i}(t), \quad \underline{b}(t) = \frac{1}{\sum_{n=0}^{N-1} c^n} \sum_{i=1}^{N} c^{N-i} b_{q_i}(t), \label{barunder_F}
\end{align*}
and $q_1q_2\cdots q_N$ is a permutation of $\{1,2,\cdots, N\}$ depending on the value $b(t)$ at time $t$ such that
\begin{equation}\label{random_order5}
b_{q_1}(t) \le b_{q_2}(t) \le \cdots \le b_{q_N}(t).
\end{equation}
Moreover, we have
\begin{equation}\label{B_size}
\eta D_b(t) \le B(t) \le D_b(t).
\end{equation}
\end{enumerate}

Note that due to the analyticity of solution to system \eqref{s_phs}, functions $Q(t), F(t), A(t)$ and $B(t)$ respectively introduced in \eqref{Q_function},\eqref{F_function}, \eqref{A_function} and \eqref{B_function} are Lipschitz continuous. Based on this fine property, we define the following two functions $\mathcal{E}_1(t)$ and $\mathcal{E}_2(t)$:
\begin{equation}\label{energy_phs}
\mathcal{E}_1(t) : = Q(t) + \frac{\eta\psi_l \sin \gamma }{2N\psi_u(\sum_{n=0}^{N-1}c^n) \gamma} m F(t) + 2m^2 A(t),
\end{equation}
\begin{equation}\label{energy_fre}
\mathcal{E}_2(t) := F(t) + \frac{\eta \psi_l}{2N\psi_u(\sum_{n=0}^{N-1}c^n)} m A(t) + 2m^2 B(t),
\end{equation}
which will help us to capture the dissipation mechanism of system.

Finally, we provide a basic estimate related to initial configuration.

\begin{lemma}
Let $(\theta(t), \omega(t))$ be a solution to system \eqref{s_phs} and assume the initial data $(\theta(0), \omega(0))$ satisfies the following condition 
\begin{equation}\label{assume_22}
2m D(\Omega) + 4mK\psi_u \sin \alpha + (1+4mK\psi_u \cos \alpha) D_\theta(0) + 3m D_\omega(0) < \beta.
\end{equation}
Then, we have
\begin{equation}\label{assume_21}
D_\theta(0) + \frac{\eta\psi_l \sin \gamma }{2N\psi_u(\sum_{n=0}^{N-1} c^n)\gamma}  m D_\omega(0) + 2m^2 D_a (0)< \beta.
\end{equation}
\end{lemma}
\begin{proof}
Recall $\omega_i(t) = \dot{\theta}_i(t)$ and $a_i(t) = \dot{\omega}_i(t)$, and we see from \eqref{s_phs} that
\begin{equation}\label{acce_initial}
\begin{aligned}
&m a_i(0) + \omega_i(0) = \Omega_i + \frac{K}{N} \sum_{j=1}^N \psi_{ij} \sin (\theta_j(0) - \theta_i(0) + \alpha),\\
& \text{i.e.} \quad a_i(0) = \frac{1}{m} \left[ - \omega_i(0) +\Omega_i + \frac{K}{N} \sum_{j=1}^N \psi_{ij} \sin (\theta_j(0) - \theta_i(0) + \alpha) \right].
\end{aligned}
\end{equation}
Moreover, we can find some indices $i$ and $j$ such that
\begin{equation*}
D_a(0) = a_i(0) - a_j(0).
\end{equation*}
%From \eqref{acce_initial}, it is easy to see that
%\begin{equation*}
%\begin{aligned}
%&a_i(0) =  \frac{1}{m} \left[ - \omega_i(0) +\Omega_i + \frac{K}{N} \sum_{k=1}^N \psi_{ik} \sin (\theta_k(0) - \theta_i(0) + \alpha) \right],
%\end{aligned}
%\end{equation*}
%\begin{equation*}
% a_j(0) = \frac{1}{m} \left[ - \omega_j(0) +\Omega_j + \frac{K}{N} \sum_{k=1}^N \psi_{jk} \sin (\theta_k(0) - \theta_j(0) + \alpha) \right].
%\end{equation*}
This yields from \eqref{acce_initial} that
\begin{equation}\label{D_acce_initial}
\begin{aligned}
D_a(0) 
%&= a_i(0) - a_j(0)\\
%&=\frac{1}{m} \left[ - \omega_i(0) +\Omega_i + \frac{K}{N} \sum_{k=1}^N \psi_{ik} \sin (\theta_k(0) - \theta_i(0) + \alpha) \right]\\
%&\quad - \frac{1}{m} \left[ - \omega_j(0) +\Omega_j + \frac{K}{N} \sum_{k=1}^N \psi_{jk} \sin (\theta_k(0) - \theta_j(0) + \alpha) \right]\\
&= -\frac{1}{m} (\omega_i(0) - \omega_j(0)) + \frac{1}{m} (\Omega_i - \Omega_j) \\
&\quad + \frac{1}{m} \frac{K}{N} \sum_{k=1}^N \left [ \psi_{ik} \sin (\theta_k(0) - \theta_i(0) + \alpha) - \psi_{jk} \sin (\theta_k(0) - \theta_j(0) + \alpha)\right]\\
&\le \frac{1}{m} D_\omega(0) + \frac{1}{m} D(\Omega) \\
&\quad + \frac{1}{m} \frac{K}{N} \sum_{k=1}^N \left[ \psi_{ik} \sin (\theta_k(0) - \theta_i(0)) \cos \alpha + \psi_{ik} \cos (\theta_k(0) - \theta_i(0)) \sin \alpha   \right.\\
&\qquad\qquad\qquad \left. - \psi_{jk} \sin (\theta_k(0) - \theta_j(0)) \cos \alpha - \psi_{jk} \cos (\theta_k(0) - \theta_j(0)) \sin \alpha \right]\\
&\le \frac{1}{m} D_\omega(0) + \frac{1}{m} D(\Omega)+\frac{2K\psi_u \cos \alpha}{m} D_\theta(0) + \frac{2K \psi_u \sin \alpha}{m}.
\end{aligned}
\end{equation}
where we applied the following facts
\begin{equation*}
\sin (x+y) = \sin x \cos y + \cos x \sin y, \quad |\sin x| \le |x|, \quad |\cos x| \le 1.
\end{equation*}
Thus, it follows from \eqref{D_acce_initial} that
\begin{equation*}
\begin{aligned}
& D_\theta(0) + \frac{\eta\psi_l \sin \gamma }{2N\psi_u(\sum_{n=0}^{N-1} c^n)\gamma}  m D_\omega(0) + 2m^2 D_a (0) \\
&\le D_\theta(0) + m D_\omega(0) \\
&\quad + 2m^2 \left[\frac{1}{m} D_\omega(0) + \frac{1}{m} D(\Omega)+\frac{2K\psi_u \cos \alpha}{m} D_\theta(0) + \frac{2K \psi_u \sin \alpha}{m} \right]\\
&= 2m D(\Omega) + 4mK\psi_u \sin \alpha + (1 + 4mK \psi_u \cos \alpha)D_\theta(0) + 3mD_\omega(0).
\end{aligned}
\end{equation*}
This together with the assumption \eqref{assume_22} yields 
\begin{equation*}
D_\theta(0) + \frac{\eta\psi_l \sin \gamma }{2N\psi_u(\sum_{i=0}^{N-1} c^i)\gamma}  m D_\omega(0) + 2m^2 D_a (0)< \beta ,
\end{equation*}
which completes the proof.
\end{proof}

\section{Entrance to a small region}\label{sec:3}
\setcounter{equation}{0}

In this section, we present that all Kuramoto oscillators will be trapped into a region confined into a quarter circle. More precisely,  we analyze the dynamics of energy function $\mathcal{E}_1(t)$ defined in \eqref{energy_phs} which can control the diameter of phase and frequency, and then provide a first-order Gronwall-type inequality of $\mathcal{E}_1(t)$. This leads to the small uniform boundedness of phase diameter after some finite time.

 We first provide the key estimate which clearly shows the dissipative mechanism in terms of the phase diameter. This will be crucially used in the latter analysis on the dynamics of $Q(t)$.

\begin{lemma}\label{phs_subadd}
Suppose oscillators' phases at time $t$ satisfy
\begin{equation*}
\theta_1(t) \le \theta_2(t) \le \cdots \le \theta_N(t), \quad D_\theta(t) < \gamma < \pi.
\end{equation*}
Then, we have
\begin{equation*}
\sum_{i=2}^N \sum_{\substack{j=1\\ (j,i) \in E}}^{i-1} \sin (\theta_j(t) - \theta_i(t)) \le -\sin D_\theta(t),\quad
\sum_{i=1}^{N-1} \sum_{\substack{j=i+1\\(j,i) \in E}}^N \sin (\theta_j(t) - \theta_i(t)) \ge \sin D_\theta(t).
\end{equation*}
\end{lemma}

As the proof of Lemma \ref{phs_subadd} is rather lengthy, we will put it in the Appendix section. Now, as the solution to system \eqref{s_phs} is analytic, we can express the time line into the union of countably many intervals as follows,
\begin{equation}\label{split-time}
[0,+\infty)=\bigcup_{l=1}^{+\infty} I_{l},\quad I_{l}=[t_{l-1},t_l),\quad t_0=0,
\end{equation}
where, in each $I_l$, the orders of oscillators' phases, frequencies, accelerations and jerks are unchanged.
Then, in each $I_l$, we can apply the key estimates in Lemma \ref{phs_subadd} to obtain a second-order differential inequality of $Q(t)$. 

\begin{lemma}\label{s_Q_dynamic}
Let $(\theta(t), \omega(t))$ be a solution to system \eqref{s_phs}, and suppose Assumption $(\mathcal{H}_2)$ holds and 
\begin{equation*}
D_\theta(t) < \gamma < \pi, \quad \text{for} \ t \in I_l,
\end{equation*}
where $I_l$ is some time interval defined in \eqref{split-time}. Then, we have
\begin{equation*}
m\ddot{Q}(t) + \dot{Q}(t) \le D(\Omega) + 2K\psi_u \sin \alpha - \frac{2K\psi _l \cos \alpha \sin \gamma}{N(\sum_{n=0}^{N-1} c^n)\gamma} Q(t), \quad t \in I_l.
\end{equation*}
\end{lemma}
\begin{remark}\label{rm3.1}
As $t \in I_l$, the order of oscillators' phases is unchanged. Then, the second order derivative of $Q(t)$ is well defined. However, at the end point of $I_l$, the second derivative of $Q(t)$ may produce delta function. Therefore, the above second order differential inequality is not enough to prove the dissipation of $Q(t)$.	
\end{remark}

\begin{proof}
For $t \in I_l$, there exists some permutation $l_1l_2\cdots l_N$ of $\{1,2,\ldots,N\}$ such that oscillators' phases at time $t$ are well-ordered as follows
\begin{equation}\label{random_order12}
\theta_{l_1}(t) \le \theta_{l_2}(t) \le \cdots \le \theta_{l_N}(t).
\end{equation}
As a matter of convenience and without loss of generality, in \eqref{random_order12}, we assume $l_i = i$, i.e., 
\begin{equation}\label{regular_order1}
\theta_1(t) \le \theta_2(t) \le \cdots \le \theta_N(t), 
\end{equation}
if necessary, we still adopt \eqref{random_order12}. Then, based on the definition of $Q(t)$ in \eqref{Q_function}, we have
\begin{equation}\label{Q_definition}
Q(t) = \bar{\theta}(t) - \underline{\theta}(t), \quad \bar{\theta}(t) = \frac{1}{\sum_{n=0}^{N-1} c^n} \sum_{i=1}^N c^{i-1} \theta_{i}(t), \quad \underline{\theta}(t) = \frac{1}{\sum_{n=0}^{N-1} c^n} \sum_{i=1}^{N} c^{N-i} \theta_{i}(t).
\end{equation}
From \eqref{Q_definition}, we see that
\begin{equation}\label{H-1}
m\ddot{Q}(t) + \dot{Q}(t) = m (\ddot{\bar{\theta}}(t) - \ddot{\underline{\theta}}(t)) + (\dot{\bar{\theta}}(t) - \dot{\underline{\theta}}(t)) = m \ddot{\bar{\theta}}(t) + \dot{\bar{\theta}}(t) - (m\ddot{\underline{\theta}}(t) + \dot{\underline{\theta}}(t)). 
\end{equation}
In the sequel, we divide the proof into three steps.

\noindent $\bullet$ {\bf Step 1:}
We first consider the dynamics of $\bar{\theta}(t)$. It is clear to see from \eqref{s_phs} and \eqref{Q_definition} that
\begin{equation*}
\begin{aligned}
m \ddot{\bar{\theta}}(t) + \dot{\bar{\theta}}(t) 
%&= m  \frac{1}{\sum_{n=0}^{N-1} c^n} \sum_{i=1}^N c^{i-1} \ddot{\theta}_{i}(t) + \frac{1}{\sum_{n=0}^{N-1} c^n} \sum_{i=1}^N c^{i-1} \dot{\theta}_{i}(t)\\
&=\frac{1}{\sum_{n=0}^{N-1} c^n} \sum_{i=1}^N c^{i-1}(m \ddot{\theta}_i(t) + \dot{\theta}_i(t)) \\
&= \frac{1}{\sum_{n=0}^{N-1} c^n} \sum_{i=1}^N c^{i-1} \left( \Omega_i + \frac{K}{N} \sum_{j=1}^N \psi_{ij} \sin (\theta_j(t) - \theta_i(t) + \alpha)\right)\\
%&= \frac{1}{\sum_{n=0}^{N-1} c^n} \sum_{i=1}^N c^{i-1} \Omega_i  + \frac{K}{N(\sum_{n=0}^{N-1} c^n)} \sum_{i=1}^N c^{i-1} \sum_{j=1}^N \psi_{ij} \sin (\theta_j(t) - \theta_i(t) + \alpha)\\
&=  \frac{1}{\sum_{n=0}^{N-1} c^n} \sum_{i=1}^N c^{i-1} \Omega_i  \\
&\quad + \frac{K}{N(\sum_{n=0}^{N-1} c^n)} \sum_{i=1}^N c^{i-1} \sum_{j=1}^N \psi_{ij} [\sin (\theta_j(t) - \theta_i(t)) \cos \alpha + \cos (\theta_j(t) - \theta_i(t)) \sin \alpha].
\end{aligned}
\end{equation*}
This yields that 
\begin{equation*}
\begin{aligned}
m \ddot{\bar{\theta}}(t) + \dot{\bar{\theta}}(t) &\le \Omega_M + K \psi_u \sin \alpha + \frac{K\cos \alpha}{N(\sum_{n=0}^{N-1} c^n)} \sum_{i=1}^N c^{i-1} \sum_{j=1}^N \psi_{ij} \sin (\theta_j(t) - \theta_i(t))\\
&\le \Omega_M + K \psi_u \sin \alpha + \frac{K\cos \alpha}{N(\sum_{n=0}^{N-1} c^n)}  \sum_{i=2}^N \sum_{j=1}^{i-1} (c^{i-1} - c^{j-1}) \psi_{ij} \sin (\theta_j(t) - \theta_i(t))\\
&\le \Omega_M + K \psi_u \sin \alpha + \frac{K\cos \alpha}{N(\sum_{n=0}^{N-1} c^n)}  \sum_{i=2}^N \sum_{j=1}^{i-1} \psi_{ij} \sin (\theta_j(t) - \theta_i(t))\\
%&=\Omega_M + K \psi_u \sin \alpha + \frac{K\cos \alpha}{N(\sum_{n=0}^{N-1} c^n)}  \sum_{i=2}^N \sum_{\substack{j=1\\ (j,i) \in E}}^{i-1} \psi_{ij} \sin (\theta_j(t) - \theta_i(t))\\
& \le \Omega_M + K \psi_u \sin \alpha + \frac{K\psi _l \cos \alpha}{N(\sum_{n=0}^{N-1} c^n)}  \sum_{i=2}^N \sum_{\substack{j=1\\ (j,i) \in E}}^{i-1} \sin (\theta_j(t) - \theta_i(t)).
\end{aligned}
\end{equation*}
Here, we applied $\psi_{ij} = \psi_{ji}$ and the constraint on $c$ in Assumption $(\mathcal{H}_2)$ yielding
\begin{equation}\label{H-1a}
\begin{aligned}
&\sum_{i=1}^N c^{i-1} \sum_{j=1}^N \psi_{ij} \sin (\theta_j(t) - \theta_i(t))\\
&=\sum_{i=1}^N c^{i-1} \sum_{\substack{j=1\\j < i}}^N \psi_{ij} \sin (\theta_j(t) - \theta_i(t)) + \sum_{i=1}^N c^{i-1} \sum_{\substack{j=1\\j > i}}^N \psi_{ij} \sin (\theta_j(t) - \theta_i(t))\\
&=\sum_{i=1}^N c^{i-1} \sum_{\substack{j=1\\j < i}}^N \psi_{ij} \sin (\theta_j(t) - \theta_i(t)) +  \sum_{j=1}^N c^{j-1} \sum_{\substack{i=1\\i > j}}^N \psi_{ji} \sin (\theta_i(t) - \theta_j(t))\\
&= \sum_{i=1}^N c^{i-1} \sum_{\substack{j=1\\j < i}}^N \psi_{ij} \sin (\theta_j(t) - \theta_i(t)) -  \sum_{j=1}^N c^{j-1} \sum_{\substack{i=1\\i > j}}^N \psi_{ji} \sin (\theta_j(t) - \theta_i(t))\\
&= \sum_{i=2}^N \sum_{j=1}^{i-1} (c^{i-1} - c^{j-1}) \psi_{ij} \sin (\theta_j(t) - \theta_i(t)),
\end{aligned}
\end{equation}
and 
\begin{equation}\label{H-1b}
i - j \ge 1, \quad c^{i-1} - c^{j-1} = c^{j-1} (c^{i-j} - 1) > 1, \quad \sin (\theta_j(t) - \theta_i(t)) \le 0, \ \text{for} \ j < i.
\end{equation}
Moreover, owing to Lemma \ref{phs_subadd}, we find
\begin{equation}\label{H-2}
\begin{aligned}
m \ddot{\bar{\theta}}(t) + \dot{\bar{\theta}}(t) 
%&\le \Omega_M + K \psi_u \sin \alpha + \frac{K\psi _l \cos \alpha}{N(\sum_{n=0}^{N-1} c^n)} \sin (\theta_1(t) - \theta_N(t))\\
&\le\Omega_M + K \psi_u \sin \alpha - \frac{K\psi _l \cos \alpha}{N(\sum_{n=0}^{N-1} c^n)} \sin D_\theta(t). 
\end{aligned}
\end{equation}

\noindent $\bullet$ {\bf Step 2:} Next, by similar argument, we can prove the following estimate of the dynamics of $\underline{\theta}(t)$,
\begin{equation}\label{H-3}
\begin{aligned}
m\ddot{\underline{\theta}}(t) + \dot{\underline{\theta}}(t) 
%&\ge \Omega_m - K \psi_u \sin \alpha + \frac{K\psi_l\cos \alpha}{N(\sum_{n=0}^{N-1} c^n)} \sin (\theta_N(t) - \theta_1(t)) \\
&\ge \Omega_m - K \psi_u \sin \alpha + \frac{K\psi_l\cos \alpha}{N(\sum_{n=0}^{N-1} c^n)} \sin D_\theta(t).
\end{aligned}
\end{equation}

\noindent $\bullet$ {\bf Step 3:} Now, we collect \eqref{H-1}, \eqref{H-2} and \eqref{H-3} to obtain
\begin{equation*}
\begin{aligned}
m\ddot{Q}(t) + \dot{Q}(t) 
%& \le \Omega_M + K \psi_u \sin \alpha - \frac{K\psi _l \cos \alpha}{N(\sum_{n=0}^{N-1} c^n)} \sin D_\theta(t)\\
%&\quad - \left( \Omega_m - K \psi_u \sin \alpha + \frac{K\psi_l\cos \alpha}{N(\sum_{n=0}^{N-1} c^n)} \sin D_\theta(t)\right)\\
&= D(\Omega) + 2K\psi_u \sin \alpha - \frac{2K\psi _l \cos \alpha}{N(\sum_{n=0}^{N-1} c^n)} \sin D_\theta(t) \\
&\le D(\Omega) + 2K\psi_u \sin \alpha - \frac{2K\psi _l \cos \alpha}{N(\sum_{n=0}^{N-1} c^n)} \frac{\sin \gamma}{\gamma} D_\theta(t)\\
&\le D(\Omega) + 2K\psi_u \sin \alpha - \frac{2K\psi _l \cos \alpha \sin \gamma}{N(\sum_{n=0}^{N-1} c^n)\gamma} Q(t), \quad t \in I_l,
\end{aligned}
\end{equation*}
Here, we applied the decreasing property of $\frac{\sin x}{x}, x \in (0,\pi]$ and the relation \eqref{phs_size}.
%Therefore, we derive the desired result.
\end{proof}

As mentioned in Remark \ref{rm3.1}, the above estimate of $Q(t)$ in Lemma \ref{s_Q_dynamic} is insufficient. The main difficulty comes from the second order derivative of $Q(t)$, which in fact is a combination of the accelerations $\ddot{\theta}_i$. Therefore, in the sequel, we present a rough estimate on the dynamics of $A(t)$, which is equivalent to the diameter of the accelerations, see defined in \eqref{A_function}. For this, we directly differentiate \eqref{s_phs} with respect to time $t$ to obtain
\begin{equation}\label{s_fre}
\begin{aligned}
m\ddot{\omega}_i(t) + \dot{\omega}_i(t) = \frac{K}{N} \sum_{j=1}^N \psi_{ij} \cos (\theta_j(t) - \theta_i(t) + \alpha) (\omega_j(t) - \omega_i(t)).
\end{aligned}
\end{equation}
Recall $a_i(t) = \dot{\omega}_i(t)$, and thus it follows from \eqref{s_fre} that
\begin{equation}\label{f_acce}
\begin{aligned}
m \dot{a}_i(t) + a_i(t) = \frac{K}{N} \sum_{j=1}^N \psi_{ij} \cos (\theta_j(t) - \theta_i(t) + \alpha) (\omega_j(t) - \omega_i(t)).
\end{aligned}
\end{equation}

\begin{lemma}\label{f_A_dynamic}
Let $(\theta(t),\omega(t))$ be a solution to system \eqref{s_phs}. Then, we have
\begin{equation}\label{f_A_eq}
m \dot{A}(t) + A(t) \le 2K \psi_u \frac{1}{\eta} F(t), \quad t \in I_l, \ l=1,2,\cdots,
\end{equation}
where $I_l$ is defined in \eqref{split-time}.
\end{lemma}
\begin{proof}
For any fixed time interval $I_l$ and $t \in I_l$ with $l=1,2,\cdots$, there exists some permutation $p_1p_2\cdots p_N$ such that oscillators' accelerations at time $t$ are in the following ordered manner
\begin{equation}\label{random_order42}
a_{p_1}(t) \le a_{p_2}(t) \le \cdots \le a_{p_N}(t).
\end{equation}
Note that the permutation $p_1p_2\cdots p_N$ here may be different from that $l_1l_2\cdots l_N$ in \eqref{random_order12} at the same instant.
Similarly, for the sake of discussion and without loss of generality, in \eqref{random_order42}, we assume $p_i = i$, i.e.,
\begin{equation}\label{regular_order4}
a_1(t) \le a_2 (t) \le \cdots \le a_N(t),
\end{equation}
if necessary, we still adopt \eqref{random_order42}.
Then based on \eqref{regular_order4} and the construction of $A(t)$ in \eqref{A_function}, we have
\begin{equation}\label{A_definition}
A(t) = \bar{a}(t) - \underline{a}(t), \quad \bar{a}(t) = \frac{1}{\sum_{n=0}^{N-1} c^n} \sum_{i=1}^N c^{i-1} a_{i}(t), \quad \underline{a}(t) = \frac{1}{\sum_{n=0}^{N-1} c^n} \sum_{i=1}^{N} c^{N-i} a_{i}(t).
\end{equation}
It follows from \eqref{A_definition} that
\begin{equation}\label{K-1}
\begin{aligned}
m\dot{A}(t) + A(t) &= m (\dot{\bar{a}}(t) - \dot{\underline{a}}(t)) + (\bar{a}(t) - \underline{a}(t))= (m\dot{\bar{a}}(t) + \bar{a}(t)) - (m\dot{\underline{a}}(t) + \underline{a}(t)).
\end{aligned}
\end{equation}
We first estimate the term $m\dot{\bar{a}}(t) + \bar{a}(t)$ in \eqref{K-1},
\begin{equation}\label{K-2}
\begin{aligned}
m\dot{\bar{a}}(t) + \bar{a}(t) 
%&= m \frac{1}{\sum_{n=0}^{N-1} c^n} \sum_{i=1}^N c^{i-1} \dot{a}_{i}(t)  + \frac{1}{\sum_{n=0}^{N-1} c^n} \sum_{i=1}^N c^{i-1} a_{i}(t)\\
&= \frac{1}{\sum_{n=0}^{N-1} c^n} \sum_{i=1}^N c^{i-1} (m \dot{a}_i(t) + a_i(t))\\
&=\frac{1}{\sum_{n=0}^{N-1} c^n} \sum_{i=1}^N c^{i-1} \left[ \frac{K}{N} \sum_{j=1}^N \psi_{ij} \cos (\theta_j(t) - \theta_i(t) + \alpha) (\omega_j(t) - \omega_i(t)) \right]\\
&\le K \psi_u D_\omega(t).
\end{aligned}
\end{equation}
%This together with Lemma \ref{size} yields that
%\begin{equation}
%\begin{aligned}
%m\dot{\bar{a}}(t) + \bar{a}(t) \le K \psi_u \frac{1}{\eta} F(t).
%\end{aligned}
%\end{equation}
Furthermore, we can apply the similar argument as in \eqref{K-2} to find
%deal with the term $m\dot{\underline{a}}(t) + \underline{a}(t)$ in \eqref{K-1},
\begin{equation}\label{K-3}
\begin{aligned}
m\dot{\underline{a}}(t) + \underline{a}(t) 
%&= m \frac{1}{\sum_{n=0}^{N-1} c^n} \sum_{i=1}^{N} c^{N-i} \dot{a}_{i}(t)+ \frac{1}{\sum_{n=0}^{N-1} c^n} \sum_{i=1}^{N} c^{N-i} a_{i}(t)\\
%&=\frac{1}{\sum_{n=0}^{N-1} c^n} \sum_{i=1}^{N} c^{N-i}(m\dot{a}_i(t) + a_i(t))\\
%&=\frac{1}{\sum_{n=0}^{N-1} c^n} \sum_{i=1}^{N} c^{N-i} \left[ \frac{K}{N} \sum_{j=1}^N \psi_{ij} \cos (\theta_j(t) - \theta_i(t) + \alpha) (\omega_j(t) - \omega_i(t))\right]\\
&\ge - K \psi_u D_\omega(t).
\end{aligned}
\end{equation}
Then, we collect \eqref{K-1}, \eqref{K-2}, \eqref{K-3} and \eqref{fre_size} to obtain
\begin{equation*}
\begin{aligned}
m\dot{A}(t) + A(t) 
%&= (m\dot{\bar{a}}(t) + \bar{a}(t)) - (m\dot{\underline{a}}(t) + \underline{a}(t))\\
&\le 2K \psi_u D_\omega(t) \le 2K \psi_u \frac{1}{\eta}F(t), \quad t \in I_l.
\end{aligned}
\end{equation*}
This derives the desired result.
\end{proof}

Now, we may use the damping term in the differential inequality of $A(t)$ to control the second order derivative of $Q(t)$. However, the dynamics of $Q(t)$ and $A(t)$ are still not dissipative, since there is an additional term $F(t)$, which is equivalent to the diameter of the frequencies. This pushes us to study the estimate on the dynamics of $F(t)$ defined in \eqref{F_function}.
Recall $\omega_i(t) = \dot{\theta}_i(t)$, then it follows from \eqref{s_phs} that
\begin{equation}\label{f_fre2} 
m\dot{\omega}_i(t) + \omega_i(t) = \Omega_i + \frac{K}{N} \sum_{j=1}^N \psi_{ij} \sin (\theta_j(t) - \theta_i(t) + \alpha), \quad t > 0.
\end{equation}

\begin{lemma}\label{f_F_dynamic}
Let $(\theta(t),\omega(t))$ be a solution to system \eqref{s_phs}. Then, we have
\begin{equation*}
m \dot{F}(t) + F(t) \le D(\Omega) + 2K \psi_u \sin \alpha + 2K \psi_u \cos \alpha \frac{1}{\eta} Q(t), \quad t \in I_l, \ l=1,2,\cdots,
\end{equation*}
where $I_l$ is defined in \eqref{split-time}.\end{lemma}
\begin{proof}
For any fixed time interval $I_l$ and $t \in I_l$ with $l=1,2,\cdots$, we can find one permutation $k_1k_2\cdots k_N$ of \{1,2,\ldots,N\} such that oscillators' frequencies at time $t$ are well-ordering as below
\begin{equation}\label{random_order32}
\omega_{k_1}(t) \le \omega_{k_2}(t) \le \cdots \le \omega_{k_N}(t).
\end{equation}
Note that the permutation $k_1k_2\cdots k_N$ here may differ from  the permutation $l_1l_2\cdots l_N$ in \eqref{random_order12} and $p_1p_2\cdots p_N$ in \eqref{random_order42} at the same instant.
For the sake of discussion, in \eqref{random_order32}, we assume $k_i = i$ without loss of generality, i.e.,
\begin{equation}\label{regular_order3}
\omega_1(t) \le \omega_2 (t) \le \cdots \le \omega_N(t),
\end{equation}
if necessary, we still adopt \eqref{random_order32}. 
Then based on \eqref{regular_order3} and the construction of $F(t)$ in \eqref{F_function}, we have
\begin{equation}\label{F_definition}
F(t) = \bar{\omega}(t) - \underline{\omega}(t), \quad \bar{\omega}(t) = \frac{1}{\sum_{n=0}^{N-1} c^n} \sum_{i=1}^N c^{i-1} \omega_{i}(t), \quad \underline{\omega}(t) = \frac{1}{\sum_{n=0}^{N-1} c^n} \sum_{i=1}^{N} c^{N-i} \omega_{i}(t).
\end{equation}
It follows from \eqref{F_definition} that
\begin{equation}\label{J-1}
\begin{aligned}
m\dot{F}(t) + F(t) &= m (\dot{\bar{\omega}}(t) - \dot{\underline{\omega}}(t)) + (\bar{\omega}(t) - \underline{\omega}(t))= (m\dot{\bar{\omega}}(t) + \bar{\omega}(t)) - (m\dot{\underline{\omega}}(t) + \underline{\omega}(t)).
\end{aligned}
\end{equation}
We first cope with the term $m\dot{\bar{\omega}}(t) + \bar{\omega}(t)$ in \eqref{J-1},
\begin{equation*}
\begin{aligned}
m\dot{\bar{\omega}}(t) + \bar{\omega}(t) 
%&= m \frac{1}{\sum_{n=0}^{N-1} c^n} \sum_{i=1}^N c^{i-1} \dot{\omega}_{i}(t) + \frac{1}{\sum_{n=0}^{N-1} c^n} \sum_{i=1}^N c^{i-1} \omega_{i}(t)\\
&=\frac{1}{\sum_{n=0}^{N-1} c^n} \sum_{i=1}^N c^{i-1} (m\dot{\omega}_i(t) + \omega_i(t))\\
&= \frac{1}{\sum_{n=0}^{N-1} c^n} \sum_{i=1}^N c^{i-1}\left[ \Omega_i + \frac{K}{N} \sum_{j=1}^N \psi_{ij} \sin (\theta_j(t) - \theta_i(t) + \alpha)\right]\\
%&=  \frac{1}{\sum_{n=0}^{N-1} c^n} \sum_{i=1}^N c^{i-1} \Omega_i +  \frac{K}{N(\sum_{n=0}^{N-1} c^n)} \sum_{i=1}^N c^{i-1} \sum_{j=1}^N \psi_{ij} \sin (\theta_j(t) - \theta_i(t) + \alpha)\\
&= \frac{1}{\sum_{n=0}^{N-1} c^n} \sum_{i=1}^N c^{i-1} \Omega_i \\
&\quad+  \frac{K}{N(\sum_{n=0}^{N-1} c^n)} \sum_{i=1}^N c^{i-1} \sum_{j=1}^N \psi_{ij} \left[\sin(\theta_j(t) - \theta_i(t)) \cos \alpha + \cos(\theta_j(t) - \theta_i(t)) \sin \alpha  \right].
\end{aligned}
\end{equation*}
Moreover, it can be estimated as below
\begin{equation}\label{J-2}
\begin{aligned}
m\dot{\bar{\omega}}(t) + \bar{\omega}(t) &\le \Omega_M + K \psi_u \sin \alpha + K \psi_u \cos\alpha D_\theta(t),
\end{aligned}
\end{equation}
where we used the following relations
\begin{equation*}
|\cos (\theta_j(t) - \theta_i(t))| \le 1,\quad |\sin (\theta_j(t) - \theta_i(t))| \le |\theta_j(t) - \theta_i(t)| \le D_\theta(t).
\end{equation*}
%Next, we deal with another term $m\dot{\underline{\omega}}(t) + \underline{\omega}(t)$ in \eqref{J-1},
%\begin{equation*}
%\begin{aligned}
%m\dot{\underline{\omega}}(t) + \underline{\omega}(t) &= m \frac{1}{\sum_{n=0}^{N-1} c^n} \sum_{i=1}^{N} c^{N-i} \dot{\omega}_{i}(t) + \frac{1}{\sum_{n=0}^{N-1} c^n} \sum_{i=1}^{N} c^{N-i} \omega_{i}(t)\\
%&= \frac{1}{\sum_{n=0}^{N-1} c^n} \sum_{i=1}^{N} c^{N-i} (m\dot{\omega}_i(t) + \omega_i(t))\\
%&=\frac{1}{\sum_{n=0}^{N-1} c^n} \sum_{i=1}^{N} c^{N-i} \left[ \Omega_i + \frac{K}{N} \sum_{j=1}^N \psi_{ij} \sin (\theta_j(t) - \theta_i(t) + \alpha)\right]\\
%&= \frac{1}{\sum_{n=0}^{N-1} c^n} \sum_{i=1}^{N} c^{N-i} \Omega_i + \frac{K}{N(\sum_{n=0}^{N-1} c^n)} \sum_{i=1}^{N} c^{N-i} \sum_{j=1}^N \psi_{ij} \sin (\theta_j(t) - \theta_i(t) + \alpha)\\
%&=\frac{1}{\sum_{n=0}^{N-1} c^n} \sum_{i=1}^{N} c^{N-i} \Omega_i \\
%&\quad +  \frac{K}{N(\sum_{n=0}^{N-1} c^n)} \sum_{i=1}^{N} c^{N-i} \sum_{j=1}^N \psi_{ij} [\sin (\theta_j(t) - \theta_i(t)) \cos \alpha + \cos (\theta_j(t) - \theta_i(t)) \sin \alpha].
%\end{aligned}
%\end{equation*}
On the other hand, we can apply the similar argument as in \eqref{J-2} to derive
\begin{equation}\label{J-3}
\begin{aligned}
m\dot{\underline{\omega}}(t) + \underline{\omega}(t) \ge \Omega_m - K \psi_u \sin \alpha - K \psi_u \cos \alpha D_\theta(t).
\end{aligned}
\end{equation}
Then, we combine \eqref{J-1}, \eqref{J-2}, \eqref{J-3} and \eqref{phs_size} to get
\begin{equation*}
\begin{aligned}
m\dot{F}(t) + F(t) 
%&= (m\dot{\bar{\omega}}(t) + \bar{\omega}(t)) - (m\dot{\underline{\omega}}(t) + \underline{\omega}(t))\\
%&\le \Omega_M + K \psi_u \sin \alpha + K \psi_u \cos\alpha D_\theta(t) - \left( \Omega_m - K \psi_u \sin \alpha - K \psi_u \cos \alpha D_\theta(t) \right)\\
&\le D(\Omega) + 2K \psi_u \sin \alpha + 2K \psi_u \cos \alpha D_\theta(t)\\
&\le D(\Omega) + 2K \psi_u \sin \alpha + 2K \psi_u \cos \alpha \frac{1}{\eta} Q(t), \quad t \in I_l.
\end{aligned}
\end{equation*}
This derives the desired result.
\end{proof}

Now, we are ready to provide a Gronwall type inequality of $\mathcal{E}_1(t)$ defined in \eqref{energy_phs}.

\begin{lemma}\label{phs_sum_dynamics}
Let $(\theta(t),\omega(t))$ be a solution to system \eqref{s_phs} and suppose Assumption $(\mathcal{H}_1)$-$(\mathcal{H}_3)$ hold. Then, we have
\begin{equation}\label{phs_sum_eq}
\frac{d}{dt} \mathcal{E}_1(t) \le 2 (D(\Omega) + 2K \psi_u \sin \alpha) -  \frac{K\psi _l \cos \alpha \sin \gamma}{N(\sum_{n=0}^{N-1} c^n)\gamma} \mathcal{E}_1(t), \quad \quad \text{a.e.} \ 0 \le t < +\infty.
\end{equation}
%\begin{equation}\label{phs_sum_eq}
%\begin{aligned}
%&\frac{d}{dt} \left( Q(t) + \frac{\eta\psi_l \sin \gamma }{2N\psi_u(\sum_{n=0}^{N-1}c^n) \gamma} m F(t) + 2m^2 A(t)\right)\\
%&\le  2 (D(\Omega) + 2K \psi_u \sin \alpha) -  \frac{K\psi _l \cos \alpha \sin \gamma}{N(\sum_{n=0}^{N-1} c^n)\gamma} \left( Q(t) + \frac{\eta\psi_l \sin \gamma }{2N\psi_u(\sum_{n=0}^{N-1}c^n) \gamma} m F(t) + 2m^2 A(t)\right).
%\end{aligned}
%\end{equation}
\end{lemma}
\begin{proof}
First, we construct a set 
\begin{equation*}
\mathcal{B} = \{ T> 0 | D_\theta(t) + \frac{\eta\psi_l \sin \gamma }{2N\psi_u(\sum_{n=0}^{N-1} c^n)\gamma}  m D_\omega(t) + 2m^2 D_a (t) < \gamma < \pi, \forall \ 0 \le  t < T\}.
\end{equation*}
Owing to \eqref{assume_21} and the continuity of functions $D_\theta(t), D_\omega(t)$ and $D_a(t)$, we can find some $\varepsilon > 0$ such that
\begin{equation*}
D_\theta(t) + \frac{\eta\psi_l \sin \gamma }{2N\psi_u(\sum_{i=0}^{N-1} c^i)\gamma}  m D_\omega(t) + 2m^2 D_a (t) < \beta < \gamma, \quad \forall \ t \in [0, \varepsilon).
\end{equation*}
This means $\varepsilon \in \mathcal{B}$ and thus the set $\mathcal{B}$ is not empty. Then, we define $T^* = \sup \mathcal{B}$. We claim that $T^* = +\infty$. Suppose by contrary, i.e., $T^* < +\infty$. This means the following assertions hold
%\begin{equation*}\label{E-1}
%D_\theta(t) + \frac{\eta\psi_l \sin \gamma }{2N\psi_u(\sum_{i=0}^{N-1} c^i)\gamma}  m D_\omega(t) + 2m^2 D_a (t) <\gamma <  \pi, \quad \forall \ 0 \le  t < T^*,
%\end{equation*}
\begin{equation}\label{E-2}
\begin{aligned}
&D_\theta(t) + \frac{\eta\psi_l \sin \gamma }{2N\psi_u(\sum_{i=0}^{N-1} c^i)\gamma}  m D_\omega(t) + 2m^2 D_a (t) <\gamma <  \pi, \quad \forall \ 0 \le  t < T^*,\\
&D_\theta(T^*) + \frac{\eta\psi_l \sin \gamma }{2N\psi_u(\sum_{i=0}^{N-1} c^i)\gamma}  m D_\omega(T^*) + 2m^2 D_a (T^*) = \gamma,
\end{aligned}
\end{equation}
which also yields the relation
\begin{equation}\label{E-3}
D_\theta(t) < \gamma < \pi, \quad \forall \  0 \le t < T^*.
\end{equation}
We may assume the orders of the oscillators' phases, frequencies, accelerations and jerks are unchanged in $ [0,T^*)$ (Otherwise, we may split $[0,T^*)$ into a union of finitely many intervals, so that the order of each state quantity of oscillators is fixed on each time interval. Then we can study the dynamics on each interval). Thus, based on \eqref{E-3} and applying Lemma \ref{s_Q_dynamic}, Lemma \ref{f_F_dynamic} and Lemma \ref{f_A_dynamic}, we have for $0 \le t < T^*$, 
\begin{align}
&m\ddot{Q}(t) + \dot{Q}(t) \le D(\Omega) + 2K\psi_u \sin \alpha - \frac{2K\psi _l \cos \alpha \sin \gamma}{N(\sum_{n=0}^{N-1} c^n)\gamma} Q(t),\label{s_Q_eq}\\
&m \dot{F} + F(t) \le D(\Omega) + 2K \psi_u \sin \alpha + 2K \psi_u \cos \alpha \frac{1}{\eta} Q(t), \label{f_F_eq}\\
&m \dot{A}(t) + A(t) \le 2K \psi_u \frac{1}{\eta} F(t).\label{f_A_eq}
\end{align}
%\begin{equation}\label{s_Q_eq}
%m\ddot{Q}(t) + \dot{Q}(t) \le D(\Omega) + 2K\psi_u \sin \alpha - \frac{2K\psi _l \cos \alpha \sin \gamma}{N(\sum_{n=0}^{N-1} c^n)\gamma} Q(t),
%\end{equation}
%
%\begin{equation}\label{f_F_eq}
%m \dot{F} + F(t) \le D(\Omega) + 2K \psi_u \sin \alpha + 2K \psi_u \cos \alpha \frac{1}{\eta} Q(t), 
%\end{equation}
%
%\begin{equation}\label{f_A_eq}
%m \dot{A}(t) + A(t) \le 2K \psi_u \frac{1}{\eta} F(t), \quad t \ge 0.
%\end{equation}
Multiplying \eqref{f_F_eq} by $\frac{\eta\psi_l \sin \gamma }{2N\psi_u(\sum_{n=0}^{N-1}c^n) \gamma}$, we get
\begin{equation}\label{L-1}
\begin{aligned}
&\frac{\eta\psi_l \sin \gamma }{2N\psi_u(\sum_{n=0}^{N-1}c^n) \gamma} m \dot{F}(t) + \frac{\eta\psi_l \sin \gamma }{2N\psi_u(\sum_{n=0}^{N-1}c^n) \gamma} F(t) \\
&\le (D(\Omega) + 2K \psi_u \sin \alpha) \frac{\eta\psi_l \sin \gamma }{2N\psi_u(\sum_{n=0}^{N-1}c^n) \gamma} + \frac{K\psi _l \cos \alpha \sin \gamma}{N(\sum_{n=0}^{N-1} c^n)\gamma} Q(t).
\end{aligned}
\end{equation}
Moreover, we multiply \eqref{f_A_eq} by $2m$ to obtain
\begin{equation}\label{L-2}
2m^2 \dot{A}(t) + 2mA(t) \le \frac{4mK \psi_u}{\eta} F(t).
\end{equation}
Then, we add \eqref{s_Q_eq}, \eqref{L-1} and \eqref{L-2} together to get 
\begin{equation*}
\begin{aligned}
&m\ddot{Q}(t) + \dot{Q}(t) + \frac{\eta\psi_l \sin \gamma }{2N\psi_u(\sum_{n=0}^{N-1}c^n) \gamma} m \dot{F}(t) + \frac{\eta\psi_l \sin \gamma }{2N\psi_u(\sum_{n=0}^{N-1}c^n)\gamma} F(t) + 2m^2 \dot{A}(t) + 2mA(t)\\
&\le D(\Omega) + 2K\psi_u \sin \alpha - \frac{2K\psi _l \cos \alpha \sin \gamma}{N(\sum_{n=0}^{N-1} c^n)\gamma} Q(t)+ \frac{4mK \psi_u}{\eta} F(t)  \\
&\quad + (D(\Omega) + 2K \psi_u \sin \alpha) \frac{\eta\psi_l \sin \gamma }{2N\psi_u(\sum_{n=0}^{N-1}c^n) \gamma} + \frac{K\psi _l \cos \alpha \sin \gamma}{N(\sum_{n=0}^{N-1} c^n)\gamma} Q(t). 
\end{aligned}
\end{equation*}
This further implies
\begin{equation*}
\begin{aligned}
&\frac{d}{dt} \left( Q(t) + \frac{\eta\psi_l \sin \gamma }{2N\psi_u(\sum_{n=0}^{N-1}c^n) \gamma} m F(t) + 2m^2 A(t)\right)\\
&\le -  m\ddot{Q}(t) - \frac{\eta\psi_l \sin \gamma }{2N\psi_u(\sum_{n=0}^{N-1}c^n) \gamma} F(t) - 2mA(t)\\
&\quad + D(\Omega) + 2K\psi_u \sin \alpha - \frac{K\psi _l \cos \alpha \sin \gamma}{N(\sum_{n=0}^{N-1} c^n)\gamma} Q(t)+ \frac{4mK \psi_u}{\eta} F(t) \\
&\quad + (D(\Omega) + 2K \psi_u \sin \alpha) \frac{\eta\psi_l \sin \gamma }{2N\psi_u(\sum_{n=0}^{N-1}c^n) \gamma} \\
&\le 2 (D(\Omega) + 2K \psi_u \sin \alpha) - \frac{K\psi _l \cos \alpha \sin \gamma}{N(\sum_{n=0}^{N-1} c^n)\gamma} Q(t) \\
&\quad- \left(\frac{\eta\psi_l \sin \gamma }{2N\psi_u(\sum_{n=0}^{N-1}c^n) \gamma} -  \frac{4mK \psi_u}{\eta}\right) F(t) - mA(t)
\end{aligned}
\end{equation*}
where we used the facts
\begin{equation*}
 \frac{\eta\psi_l \sin \gamma }{2N\psi_u(\sum_{n=0}^{N-1}c^n) \gamma} < 1, \quad |\ddot{Q}(t)| \le A(t).
\end{equation*}
Moreover, under the assumption \eqref{assume_3} yielding
\begin{equation*}
\begin{aligned}
&mK < \frac{\eta^2\psi_l \sin \gamma }{16N\psi_u^2(\sum_{n=0}^{N-1}c^n) \gamma} \quad \Longrightarrow \quad \frac{\eta\psi_l \sin \gamma }{4N\psi_u(\sum_{n=0}^{N-1}c^n) \gamma} -  \frac{4mK \psi_u}{\eta} > 0 ,\\
&mK < \frac{N(\sum_{n=0}^{N-1}c^n) \gamma}{2 \psi_l \cos \alpha \sin \gamma}  \quad \Longrightarrow \quad    \frac{4K \psi_u \cos \alpha}{\eta} \cdot \frac{\eta\psi_l \sin \gamma }{2N\psi_u(\sum_{n=0}^{N-1}c^n) \gamma} m = mK \frac{2 \psi_l \cos \alpha \sin \gamma}{N(\sum_{n=0}^{N-1}c^n)  \gamma} < 1 ,\\
&mK < \frac{N (\sum_{n=0}^{N-1} c^n)\gamma}{2\psi_l \cos \alpha \sin \gamma} \quad \Longrightarrow \quad    \frac{K \psi_l  \cos \alpha \sin \gamma}{N (\sum_{n=0}^{N-1} c^n) \gamma} 2m = mK \frac{2\psi_l \cos \alpha \sin \gamma}{N (\sum_{n=0}^{N-1} c^n)\gamma} < 1  ,
\end{aligned}
\end{equation*}
we have for $0 \le t < T^*$,
\begin{equation*}
\begin{aligned}
&\frac{d}{dt} \left( Q(t) + \frac{\eta\psi_l \sin \gamma }{2N\psi_u(\sum_{n=0}^{N-1}c^n) \gamma} m F(t) + 2m^2 A(t)\right)\\
&\le 2 (D(\Omega) + 2K \psi_u \sin \alpha) - \frac{K\psi _l \cos \alpha \sin \gamma}{N(\sum_{n=0}^{N-1} c^n)\gamma} Q(t) - \frac{\eta\psi_l \sin \gamma }{4N\psi_u(\sum_{n=0}^{N-1}c^n) \gamma}  F(t) - mA(t)\\
&= 2 (D(\Omega) + 2K \psi_u \sin \alpha)-  \frac{K\psi _l \cos \alpha \sin \gamma}{N(\sum_{n=0}^{N-1} c^n)\gamma} \left(Q(t) + \frac{\eta}{4K \psi_u \cos \alpha}F(t) + \frac{N (\sum_{n=0}^{N-1} c^n) \gamma}{K \psi_l  \cos \alpha \sin \gamma}mA(t) \right)\\
&\le  2 (D(\Omega) + 2K \psi_u \sin \alpha) -  \frac{K\psi _l \cos \alpha \sin \gamma}{N(\sum_{n=0}^{N-1} c^n)\gamma} \left( Q(t) + \frac{\eta\psi_l \sin \gamma }{2N\psi_u(\sum_{n=0}^{N-1}c^n) \gamma} m F(t) + 2m^2 A(t)\right).
\end{aligned}
\end{equation*}
%This completes the proof of this lemma.
According to \eqref{energy_phs}, for $0 \le t < T^*$, we further have
\begin{equation}\label{L-3}
\begin{aligned}
\frac{d}{dt} \mathcal{E}_1(t) &\le 2 (D(\Omega) + 2K \psi_u \sin \alpha) -  \frac{K\psi _l \cos \alpha \sin \gamma}{N(\sum_{n=0}^{N-1} c^n)\gamma} \mathcal{E}_1(t)\\
&= -  \frac{K\psi _l \cos \alpha \sin \gamma}{N(\sum_{n=0}^{N-1} c^n)\gamma} \left(\mathcal{E}_1(t) - \frac{2 (D(\Omega) + 2K \psi_u \sin \alpha)N(\sum_{n=0}^{N-1} c^n)\gamma}{K\psi _l \cos \alpha \sin \gamma}\right).
\end{aligned}
\end{equation}
This implies that for $0 \le t < T^*$,
\begin{equation}\label{M-1}
\mathcal{E}_1(t) \le \max \left\{ \mathcal{E}_1(0), \frac{4 (D(\Omega) + 2K \psi_u \sin \alpha)N(\sum_{n=0}^{N-1} c^n)\gamma}{K\psi _l \cos \alpha \sin \gamma} \right\} .
\end{equation}
Moreover, owing to \eqref{assume_21}, \eqref{phs_size}, \eqref{fre_size} and \eqref{acce_size}, we have
\begin{equation}\label{M-2}
\begin{aligned}
\mathcal{E}_1(0) &= Q(0) + \frac{\eta\psi_l \sin \gamma }{2N\psi_u(\sum_{n=0}^{N-1}c^n) \gamma} m F(0) + 2m^2 A(0)\\
&\le D_\theta(0) + \frac{\eta\psi_l \sin \gamma }{2N\psi_u(\sum_{i=0}^{N-1} c^i)\gamma}  m D_\omega(0) + 2m^2 D_a (0)< \beta .
\end{aligned}
\end{equation}
%\begin{equation*}
%\frac{4 (D(\Omega) + 2K \psi_u \sin \alpha)N(\sum_{n=0}^{N-1} c^n)\gamma}{K\psi _l \cos \alpha \sin \gamma} < \beta
%\end{equation*}
Due to assumptions \eqref{assume_1} and \eqref{assume_4}, we see that
\begin{equation}\label{M-2a}
\frac{4 (D(\Omega) + 2K \psi_u \sin \alpha)N(\sum_{n=0}^{N-1} c^n)\gamma}{K\psi _l \cos \alpha \sin \gamma} < \eta D^\infty < \beta.
\end{equation}
Thus, we combine \eqref{M-1}, \eqref{M-2} and \eqref{M-2a} to derive
\begin{equation*}
\mathcal{E}_1(t) < \beta, \quad \text{for} \ 0 \le t < T^*.
\end{equation*}
This together with the assumption $\eta> \frac{\beta}{\gamma}$ in  \eqref{assume_1} yields
\begin{equation*}
\begin{aligned}
&D_\theta(t) + \frac{\eta\psi_l \sin \gamma }{2N\psi_u(\sum_{n=0}^{N-1}c^n) \gamma} m D_\omega(t) + 2m^2 D_a(t) \\
&\le \frac{1}{\eta} \left( Q(t) + \frac{\eta\psi_l \sin \gamma }{2N\psi_u(\sum_{n=0}^{N-1}c^n) \gamma} m F(t) + 2m^2 A(t) \right) = \frac{\mathcal{E}_1(t)}{\eta} < \frac{\beta}{\eta} < \gamma, \quad 0 \le t < T^*.
\end{aligned}
\end{equation*}
Moreover, it follows that
\begin{equation*}
\begin{aligned}
&D_\theta(T^*) + \frac{\eta\psi_l \sin \gamma }{2N\psi_u(\sum_{n=0}^{N-1}c^n) \gamma} m D_\omega(T^*) + 2m^2 D_a(T^*) \\
&= \lim_{t \to T^*} \left( D_\theta(t) + \frac{\eta\psi_l \sin \gamma }{2N\psi_u(\sum_{n=0}^{N-1}c^n) \gamma} m D_\omega(t) + 2m^2 D_a(t)\right) \le \frac{\beta}{\eta} < \gamma,
\end{aligned}
\end{equation*}
which obviously contradicts to $\eqref{E-2}_2$. Thus, we have $T^* = +\infty$. This means
\begin{equation*}
D_\theta(t) + \frac{\eta\psi_l \sin \gamma }{2N\psi_u(\sum_{n=0}^{N-1} c^n)\gamma}  m D_\omega(t) + 2m^2 D_a (t) <\gamma <  \pi, \quad \forall \ 0 \le  t < +\infty,
\end{equation*}
which further leads to
\begin{equation}\label{M-3}
D_\theta(t) < \gamma < \pi, \quad \forall \ 0 \le t < +\infty.
\end{equation}
Then based on the fact \eqref{M-3} and the Lipschitz continuity of $\mathcal{E}_1(t)$, we can apply the similar argument as in \eqref{L-3} and conclude 
\begin{equation*}
\begin{aligned}
\frac{d}{dt} \mathcal{E}_1(t) &\le 2 (D(\Omega) + 2K \psi_u \sin \alpha) -  \frac{K\psi _l \cos \alpha \sin \gamma}{N(\sum_{n=0}^{N-1} c^n)\gamma} \mathcal{E}_1(t), \quad \text{a.e.} \ 0 \le t < +\infty.
\end{aligned}
\end{equation*}
%\begin{equation}\label{M-4}
%\begin{aligned}
%\frac{d}{dt} \mathcal{E}_1(t) &\le 2 (D(\Omega) + 2K \psi_u \sin \alpha) -  \frac{K\psi _l \cos \alpha \sin \gamma}{N(\sum_{n=0}^{N-1} c^n)\gamma} \mathcal{E}_1(t)\\
%&= -  \frac{K\psi _l \cos \alpha \sin \gamma}{N(\sum_{n=0}^{N-1} c^n)\gamma} \left(\mathcal{E}_1(t) - \frac{2 (D(\Omega) + 2K \psi_u \sin \alpha)N(\sum_{n=0}^{N-1} c^n)\gamma}{K\psi _l \cos \alpha \sin \gamma}\right), \quad 0 \le t < +\infty.
%\end{aligned}
%\end{equation}
This completes the proof.
\end{proof}

Lastly, we show that all oscillators will enter into an arc confined in a quarter circle in finite time and stay there afterwards.
\begin{lemma}\label{small}
Let $(\theta(t),\omega(t))$ be a solution to system \eqref{s_phs} and suppose Assumption $(\mathcal{H}_1)$-$(\mathcal{H}_3)$ hold. Then, there exists time $t_* > 0$ such that
\begin{equation}\label{phs_bound}
\begin{aligned}
D_\theta(t) < D^\infty < \frac{\pi}{2}, \quad t \ge t_*.
\end{aligned}
\end{equation}
Moreover, we have
\begin{equation}\label{F_bound}
\begin{aligned}
&F(t) \le \frac{8 (D(\Omega) + 2K \psi_u \sin \alpha)N^2\psi_u(\sum_{n=0}^{N-1} c^n)^2\gamma^2 }{mK \eta \psi _l^2 \cos \alpha \sin^2 \gamma}, \quad t \ge t_*.
\end{aligned}
\end{equation}
%\begin{equation}\label{fre_bound}
%D_\omega(t) \le \frac{8 (D(\Omega) + 2K \psi_u \sin \alpha)N^2\psi_u(\sum_{n=0}^{N-1} c^n)^2\gamma^2 }{mK \eta^2 \psi _l^2 \cos \alpha \sin^2 \gamma}, \quad t \ge t_*.
%\end{equation}
\end{lemma}
\begin{proof}
According to \eqref{phs_sum_eq}, we see that
\begin{equation}\label{M-4}
\begin{aligned}
\frac{d}{dt} \mathcal{E}_1(t) &\le 2 (D(\Omega) + 2K \psi_u \sin \alpha) -  \frac{K\psi _l \cos \alpha \sin \gamma}{N(\sum_{n=0}^{N-1} c^n)\gamma} \mathcal{E}_1(t),\\
&= -  \frac{K\psi _l \cos \alpha \sin \gamma}{N(\sum_{n=0}^{N-1} c^n)\gamma} \left(\mathcal{E}_1(t) - \frac{2 (D(\Omega) + 2K \psi_u \sin \alpha)N(\sum_{n=0}^{N-1} c^n)\gamma}{K\psi _l \cos \alpha \sin \gamma}\right), \ \text{a.e.} \ 0 \le t < +\infty.\\
\end{aligned}
\end{equation}
We consider two cases below.

\noindent $\bullet$ {\bf Case 1:} If 
\begin{equation*}
\mathcal{E}_1(0) > \frac{4 (D(\Omega) + 2K \psi_u \sin \alpha)N(\sum_{n=0}^{N-1} c^n)\gamma}{K\psi _l \cos \alpha \sin \gamma},
\end{equation*}
then for $\mathcal{E}_1(t) \in [\frac{4 (D(\Omega) + 2K \psi_u \sin \alpha)N(\sum_{n=0}^{N-1} c^n)\gamma}{K\psi _l \cos \alpha \sin \gamma},\mathcal{E}_1(0)]$, it follows from \eqref{M-4} that 
\begin{equation*}
\begin{aligned}
\frac{d}{dt} \mathcal{E}_1(t) 
&\le -  \frac{K\psi _l \cos \alpha \sin \gamma}{N(\sum_{n=0}^{N-1} c^n)\gamma} \left(\mathcal{E}_1(t) - \frac{2 (D(\Omega) + 2K \psi_u \sin \alpha)N(\sum_{n=0}^{N-1} c^n)\gamma}{K\psi _l \cos \alpha \sin \gamma}\right)\\
%&\le -  \frac{K\psi _l \cos \alpha \sin \gamma}{N(\sum_{n=0}^{N-1} c^n)\gamma} \left(\frac{4 (D(\Omega) + 2K \psi_u \sin \alpha)N(\sum_{n=0}^{N-1} c^n)\gamma}{K\psi _l \cos \alpha \sin \gamma} - \frac{2 (D(\Omega) + 2K \psi_u \sin \alpha)N(\sum_{n=0}^{N-1} c^n)\gamma}{K\psi _l \cos \alpha \sin \gamma}\right)\\
&\le  -  \frac{K\psi _l \cos \alpha \sin \gamma}{N(\sum_{n=0}^{N-1} c^n)\gamma} \cdot \frac{2 (D(\Omega) + 2K \psi_u \sin \alpha)N(\sum_{n=0}^{N-1} c^n)\gamma}{K\psi _l \cos \alpha \sin \gamma}\\
&= - 2 (D(\Omega) + 2K \psi_u \sin \alpha) < 0.
\end{aligned}
\end{equation*}
This yields
\begin{equation}\label{M-5}
\mathcal{E}_1(t) \le \frac{4 (D(\Omega) + 2K \psi_u \sin \alpha)N(\sum_{n=0}^{N-1} c^n)\gamma}{K\psi _l \cos \alpha \sin \gamma} , \quad \text{for} \ t \ge t_*,
\end{equation}
where $t_*$ satisfies 
\begin{equation*}
t_* = \frac{\mathcal{E}_1(0) - \frac{4 (D(\Omega) + 2K \psi_u \sin \alpha)N(\sum_{n=0}^{N-1} c^n)\gamma}{K\psi _l \cos \alpha \sin \gamma}}{2 (D(\Omega) + 2K \psi_u \sin \alpha)} \le \frac{\mathcal{E}_1(0)}{2 (D(\Omega) + 2K \psi_u \sin \alpha)} < \frac{\beta}{2 (D(\Omega) + 2K \psi_u \sin \alpha)}.
\end{equation*}

\noindent $\bullet$ {\bf Case 2:} If 
\begin{equation*}
\mathcal{E}_1(0) \le \frac{4 (D(\Omega) + 2K \psi_u \sin \alpha)N(\sum_{n=0}^{N-1} c^n)\gamma}{K\psi _l \cos \alpha \sin \gamma},
\end{equation*}
then we see from \eqref{M-4} that
\begin{equation}\label{M-6}
\mathcal{E}_1(t) \le \frac{4 (D(\Omega) + 2K \psi_u \sin \alpha)N(\sum_{n=0}^{N-1} c^n)\gamma}{K\psi _l \cos \alpha \sin \gamma}, \quad \text{for} \ t \ge 0.
\end{equation}
Hence, we set $t_* = 0$ in this case.

Then, we combine \eqref{M-5} and \eqref{M-6} to obtain
\begin{equation*}
\begin{aligned}
\mathcal{E}_1(t)\le \frac{4 (D(\Omega) + 2K \psi_u \sin \alpha)N(\sum_{n=0}^{N-1} c^n)\gamma}{ K \psi _l \cos \alpha \sin \gamma}, \quad t \ge t_*.
\end{aligned}
\end{equation*}
This further yields from \eqref{energy_phs} that
\begin{equation}\label{M-7}
F(t) \le \frac{8 (D(\Omega) + 2K \psi_u \sin \alpha)N^2\psi_u(\sum_{n=0}^{N-1} c^n)^2\gamma^2 }{mK \eta \psi _l^2 \cos \alpha \sin^2 \gamma}, \quad t \ge t_*.
\end{equation}
Moreover, due to assumption \eqref{assume_4}, we have 
\begin{equation*}
\begin{aligned}
&D_\theta(t) + \frac{\eta\psi_l \sin \gamma }{2N\psi_u(\sum_{i=0}^{N-1} c^i)\gamma}  m D_\omega(t) + 2m^2 D_a (t) \\
&\le \frac{\mathcal{E}_1(t)}{\eta}\le \frac{4 (D(\Omega) + 2K \psi_u \sin \alpha)N(\sum_{n=0}^{N-1} c^n)\gamma}{ K \eta \psi _l \cos \alpha \sin \gamma} &< D^\infty < \frac{\pi}{2}, \quad \text{for} \ t \ge t_*,
\end{aligned}
\end{equation*}
which immediately implies
\begin{equation}\label{M-8}
D_\theta(t) < D^\infty < \frac{\pi}{2}, \quad t \ge t_*.
\end{equation}
Therefore, the desired result follows from \eqref{M-8} and \eqref{M-7}.

\end{proof}

\section{Exponential frequency synchronization}\label{sec:4}
\setcounter{equation}{0}
In this section, we present that the frequency synchronization emerges exponentially fast for system \eqref{s_phs}. Similar to Section \ref{sec:3}, we investigate the dynamics of energy function $\mathcal{E}_2(t)$ defined in \eqref{energy_fre} which can control the frequency diameter, and then provide a dissipative differential inequality of $\mathcal{E}_2(t)$ resulting in the exponential synchronization. 

For this,
we consider system \eqref{s_fre} starting from $t_*$ mentioned in Lemma \ref{small}, i.e.,
\begin{equation}\label{s_fre2}
\begin{aligned}
m\ddot{\omega}_i(t) + \dot{\omega}_i(t) = \frac{K}{N} \sum_{j=1}^N \psi_{ij} \cos (\theta_j(t) - \theta_i(t) + \alpha) (\omega_j(t) - \omega_i(t)), \quad t \ge t_*,
\end{aligned}
\end{equation}
Similarly, owing to the analyticity of solution, we divide the time interval $[t_*,+\infty)$ into a union of countable intervals:
\begin{equation}\label{split-time2}
[t_*,+\infty) = \bigcup_{l=1}^{+\infty} J_l, \quad J_l = [s_{l-1}, s_l), \quad s_0 = t_*,
\end{equation}
so that the orders of oscillators' phases, frequencies, accelerations and jerks are all fixed on each time interval $J_l$. 
%In order to catch the dissipation of frequency diameter $D_\omega(t)$, we switch to investigate the dynamics of $F(t)$ defined in \eqref{F_function}  for any time $t \ge t_*$.
Parallel to Lemma \ref{phs_subadd}, we present the following estimate which will be used in the analysis on the dynamics of $F(t)$. As the proof is very similar to Lemma \ref{phs_subadd}, we omit the details. 
\begin{lemma}\label{fre_subadd}
Let $(\theta(t),\omega(t))$ be a solution to system \eqref{s_phs} and suppose the order of oscillators' frequencies at time $t$ satisfies
\begin{equation*}
\omega_1(t) \le \omega_2(t) \le \cdots \le \omega_N(t).
\end{equation*}
Then, we have
\begin{equation*}
\sum_{i=2}^N \sum_{\substack{j=1\\(j,i) \in E}}^{i-1} (\omega_j(t) - \omega_i(t))  \le -D_\omega(t),\quad \sum_{i=1}^{N-1} \sum_{\substack{j=i+1 \\ (j,i) \in E}}^N (\omega_j(t) - \omega_i(t)) \ge D_\omega(t).
\end{equation*}
\end{lemma}
Now, we intend to prove the emergence of synchronization. In Lemma \ref{small}, we already prove that all oscillators will stay in a small region after $t_*$. Then we have the following second order differential inequality of $F(t)$.
\begin{lemma}\label{s_fre_dynamic}
Let $(\theta(t), \omega(t))$ be a solution to system \eqref{s_phs} and suppose Assumption $(\mathcal{H}_1)$-$(\mathcal{H}_3)$  hold.Then, we have
\begin{equation}\label{s_fre_eq}
m\ddot{F}(t) + \dot{F}(t) \le - \frac{2K \psi_l}{N(\sum_{n=0}^{N-1} c^n)} F(t), \quad t\in J_l, \ l=1,2,\cdots,
\end{equation}
where $J_l$ is defined in \eqref{split-time2}.
\end{lemma}
\begin{proof}
For any fixed time interval $J_l$ and $t\in J_l$ with $l=1,2,\cdots$, the second order derivative is well defined. Moreover, there exists one permutation $k_1k_2\cdots k_N$ of \{1,2,\ldots,N\} such that oscillators' frequencies at time $t$ are well-ordering as below
\begin{equation}\label{random_order33}
\omega_{k_1}(t) \le \omega_{k_2}(t) \le \cdots \le \omega_{k_N}(t).
\end{equation}
%Note that the permutation $k_1k_2\cdots k_N$ here may differ from  the permutation $l_1l_2\cdots l_N$ in \eqref{random_order12} at the same time $t$.
For the convenience of discussion, in \eqref{random_order33}, we assume $k_i = i$ without loss of generality, i.e.,
\begin{equation}\label{regular_order31}
\omega_1(t) \le \omega_2 (t) \le \cdots \le \omega_N(t),
\end{equation}
if necessary, we still adopt \eqref{random_order33}.
Then based on \eqref{regular_order31} and the construction of $F(t)$ in \eqref{F_function}, we have
\begin{equation}\label{F_definition2}
F(t) = \bar{\omega}(t) - \underline{\omega}(t), \quad \bar{\omega}(t) = \frac{1}{\sum_{n=0}^{N-1} c^n} \sum_{i=1}^N c^{i-1} \omega_{i}(t), \quad \underline{\omega}(t) = \frac{1}{\sum_{n=0}^{N-1} c^n} \sum_{i=1}^{N} c^{N-i} \omega_{i}(t).
\end{equation}
It's easy to see that
\begin{equation}\label{N-1}
m\ddot{F}(t) + \dot{F}(t) = m(\ddot{\bar{\omega}}(t) - \ddot{\underline{\omega}}(t)) + (\dot{\bar{\omega}}(t) - \dot{\underline{\omega}}(t)) = (m\ddot{\bar{\omega}}(t)+\dot{\bar{\omega}}(t)) - (m\ddot{\underline{\omega}}(t) + \dot{\underline{\omega}}(t)).
\end{equation}
We first estimate the dynamics of $\bar{\omega}(t)$. It follows from \eqref{F_definition2} and \eqref{s_fre2} that
\begin{equation}\label{N-2}
\begin{aligned}
m\ddot{\bar{\omega}}(t)+\dot{\bar{\omega}}(t) &= m \frac{1}{\sum_{n=0}^{N-1} c^n} \sum_{i=1}^N c^{i-1} \ddot{\omega}_{i}(t) + \frac{1}{\sum_{n=0}^{N-1} c^n} \sum_{i=1}^N c^{i-1} \dot{\omega}_{i}(t)\\
&= \frac{1}{\sum_{n=0}^{N-1} c^n} \sum_{i=1}^N c^{i-1} (m\ddot{\omega}_{i}(t) + \dot{\omega}(t)) \\
%&= \frac{1}{\sum_{n=0}^{N-1} c^n} \sum_{i=1}^N c^{i-1} \frac{K}{N} \sum_{j=1}^N \psi_{ij} \cos (\theta_j(t) - \theta_i(t) + \alpha) (\omega_j(t) - \omega_i(t))\\
&=\frac{K}{N(\sum_{n=0}^{N-1} c^n)} \sum_{i=1}^N c^{i-1}\sum_{j=1}^N \psi_{ij} \cos (\theta_j(t) - \theta_i(t) + \alpha) (\omega_j(t) - \omega_i(t)).
\end{aligned}
\end{equation}
%&=  \frac{K}{N(\sum_{n=0}^{N-1} c^n)} \sum_{i=1}^N c^{i-1}\sum_{j=1}^N \psi_{ij} \cos (\theta_j(t) - \theta_i(t) + \alpha) (\omega_j(t) - \omega_i(t))\\
Moreover, we find that
\begin{equation}\label{N-3}
\begin{aligned}
&\sum_{i=1}^N c^{i-1}\sum_{j=1}^N \psi_{ij} \cos (\theta_j(t) - \theta_i(t) + \alpha) (\omega_j(t) - \omega_i(t))\\
&=  \sum_{i=1}^N c^{i-1}\sum_{\substack{j=1\\ j < i}}^N \psi_{ij} \cos (\theta_j(t) - \theta_i(t) + \alpha) (\omega_j(t) - \omega_i(t)) \\
&\quad+\sum_{i=1}^N c^{i-1}\sum_{\substack{j=1\\ j > i}}^N \psi_{ij} \cos (\theta_j(t) - \theta_i(t) + \alpha) (\omega_j(t) - \omega_i(t)) \\
%&=\sum_{i=1}^N c^{i-1}\sum_{\substack{j=1\\ j < i}}^N \psi_{ij} \cos (\theta_j(t) - \theta_i(t) + \alpha) (\omega_j(t) - \omega_i(t)) \\
%&\quad +  \sum_{j=1}^N c^{j-1}\sum_{\substack{i=1\\ i > j}}^N \psi_{ji} \cos (\theta_i(t) - \theta_j(t) + \alpha) (\omega_i(t) - \omega_j(t)) \\
%&\le \frac{K}{N(\sum_{n=0}^{N-1} c^n)} \sum_{i=1}^N c^{i-1}\sum_{\substack{j=1\\ j < i}}^N \psi_{ij} \cos (D^\infty + \alpha) (\omega_j(t) - \omega_i(t))\\
%&\quad + \frac{K}{N(\sum_{n=0}^{N-1} c^n)} \sum_{j=1}^N c^{j-1}\sum_{\substack{i=1\\ i > j}}^N \psi_{ji} (\omega_i(t) - \omega_j(t))\\
&\le\sum_{i=1}^N \sum_{\substack{j=1\\ j < i}}^N c^{i-1} \cos (D^\infty + \delta) \psi_{ij}(\omega_j(t) - \omega_i(t))-   \sum_{j=1}^N \sum_{\substack{i=1\\ i > j}}^N c^{j-1}\psi_{ji} (\omega_j(t) - \omega_i(t))\\
&= \sum_{i=2}^N \sum_{j=1}^{i-1} [c^{i-1} \cos (D^\infty + \delta) - c^{j-1}] \psi_{ij} (\omega_j(t) - \omega_i(t))\le  \sum_{i=2}^N \sum_{j=1}^{i-1} \psi_{ij} (\omega_j(t) - \omega_i(t)).
\end{aligned}
\end{equation}
where we applied Lemma \ref{small}, \eqref{assume_1} and the following estimates
\begin{equation*}
\begin{aligned}
&|\theta_j(t) - \theta_i(t) + \alpha| \le D_\theta(t) + \alpha \le D^\infty + \delta < \frac{\pi}{2}, \quad \cos (\theta_j(t) - \theta_i(t) + \alpha) \ge \cos (D^\infty+ \delta) > 0. \\
%& \cos (\theta_j(t) - \theta_i(t) + \alpha) \le 1.
&c^{i-1} \cos (D^\infty + \delta) - c^{j-1} = c^{j-1}[c^{i-j} \cos (D^\infty + \delta) -1] \ge c\cos (D^\infty + \delta) -1  > 1.
\end{aligned}
\end{equation*}
%\begin{equation}\label{N-4}
%\begin{aligned}
%&\frac{K}{N(\sum_{n=0}^{N-1} c^n)} \sum_{i=1}^N c^{i-1}\sum_{j=1}^N \psi_{ij} \cos (\theta_j(t) - \theta_i(t) + \alpha) (\omega_j(t) - \omega_i(t)) \\
%&\le  \frac{K}{N(\sum_{n=0}^{N-1} c^n)}\sum_{i=2}^N \sum_{j=1}^{i-1} \psi_{ij} (\omega_j(t) - \omega_i(t)).
%\end{aligned}
%\end{equation}
Then, we collect \eqref{N-2}, \eqref{N-3} and Lemma \ref{fre_subadd} to get
\begin{equation}\label{N-5}
\begin{aligned}
m\ddot{\bar{\omega}}(t)+\dot{\bar{\omega}}(t) &\le \frac{K}{N(\sum_{n=0}^{N-1} c^n)}\sum_{i=2}^N \sum_{j=1}^{i-1} \psi_{ij} (\omega_j(t) - \omega_i(t))\\
&= \frac{K}{N(\sum_{n=0}^{N-1} c^n)}\sum_{i=2}^N \sum_{\substack{j=1\\(j,i) \in E}}^{i-1} \psi_{ij} (\omega_j(t) - \omega_i(t)) \\
&\le  \frac{K\psi_l}{N(\sum_{n=0}^{N-1} c^n)}\sum_{i=2}^N \sum_{\substack{j=1\\(j,i) \in E}}^{i-1} (\omega_j(t) - \omega_i(t))\le -\frac{K\psi_l}{N(\sum_{n=0}^{N-1} c^n)} D_\omega(t).
\end{aligned}
\end{equation}
For the dynamics of $\underline{\omega}(t)$, we can apply the similar argument as in \eqref{N-5} to obtain
\begin{equation}\label{N-6}
%m\ddot{\underline{\omega}}(t) + \dot{\underline{\omega}}(t) \ge \frac{K\psi_l}{N(\sum_{n=0}^{N-1} c^n)} \sum_{i=1}^{N-1} \sum_{\substack{j=i+1 \\ (j,i) \in E}}^N (\omega_j(t) - \omega_i(t)) \ge \frac{K\psi_l}{N(\sum_{n=0}^{N-1} c^n)} D_\omega(t).
m\ddot{\underline{\omega}}(t) + \dot{\underline{\omega}}(t) \ge \frac{K\psi_l}{N(\sum_{n=0}^{N-1} c^n)} D_\omega(t).
\end{equation}
Thus, we combine \eqref{N-1}, \eqref{N-5} and \eqref{N-6} to obtain
\begin{equation*}
\begin{aligned}
m\ddot{F}(t) + \dot{F}(t)  &= (m\ddot{\bar{\omega}}(t)+\dot{\bar{\omega}}(t)) - (m\ddot{\underline{\omega}}(t) + \dot{\underline{\omega}}(t)) \\
&\le -\frac{2K\psi_l}{N(\sum_{n=0}^{N-1} c^n)} D_\omega(t) \le  -\frac{2K\psi_l}{N(\sum_{n=0}^{N-1} c^n)}F(t), \quad t\in J_l,
\end{aligned}
\end{equation*}
which yields the desired result.
\end{proof}

In the sequel, we directly differentiate \eqref{f_acce} with respect to time $t$ to obtain
\begin{equation}\label{s_acce}
\begin{aligned}
m \ddot{a}_i(t) + \dot{a}_i(t) &= \frac{K}{N} \sum_{j=1}^N \psi_{ij} \left[ \cos (\theta_j(t) - \theta_i(t) + \alpha) (a_j(t) - a_i(t)) \right.\\
&\qquad \qquad \qquad \left.- \sin (\theta_j(t) - \theta_i(t) + \alpha) (\omega_j(t) - \omega_i(t))^2\right]
\end{aligned}
\end{equation}
Recall $b_i(t) = \dot{a}_i(t)$. Then, it follows from \eqref{s_acce} that
\begin{equation}\label{f_b}
\begin{aligned}
m\dot{b}_i(t) + b_i(t) &= \frac{K}{N} \sum_{j=1}^N \psi_{ij} \left[ \cos (\theta_j(t) - \theta_i(t) + \alpha) (a_j(t) - a_i(t)) \right.\\
&\qquad \qquad \qquad \left.- \sin (\theta_j(t) - \theta_i(t) + \alpha) (\omega_j(t) - \omega_i(t))^2\right].
\end{aligned}
\end{equation}
We present a rough estimate on the dynamics of $B(t)$ in the following lemma.

\begin{lemma}\label{f_B_dynamic}
Let $(\theta(t), \omega(t))$ be a solution to \eqref{s_phs} and suppose Assumption $(\mathcal{H}_1)$-$(\mathcal{H}_3)$ hold. Then, we have
\begin{equation}\label{f_B_eq}
m \dot{B}(t) + B(t) \le \frac{2K \psi_u}{\eta} A(t) + \frac{2K\psi_u(D^\infty\cos \alpha + \sin \alpha)}{\eta^2}  F^2(t), \quad t\in J_l, \ l=1,2,\cdots,
\end{equation}
where $J_l$ is defined in \eqref{split-time2}.
\end{lemma}
\begin{proof}
For any fixed time interval $J_l$ and $t\in J_l$ with $l=1,2,\cdots$, we can find one permutation $q_1q_2\cdots q_N$ such that oscillators' jerks at time $t$ are in a well-ordered manner,
\begin{equation}\label{random_order52}
b_{q_1}(t) \le b_{q_2}(t) \le \cdots \le b_{q_N}(t).
\end{equation}
Note that the permutation $q_1q_2\cdots q_N$ here may be different from that $k_1k_2\cdots k_N$ in \eqref{random_order33} 
%and $p_1p_2\cdots p_N$ in \eqref{random_order42} 
at the same instant.
Similarly, for convenience and without loss of generality, in \eqref{random_order42}, we assume $q_i = i$, i.e.,

\begin{equation}\label{regular_order5}
b_1(t) \le b_2 (t) \le \cdots \le b_N(t),
\end{equation}
if necessary, we still adopt \eqref{random_order52}.
Then based on \eqref{regular_order5} and the construction of $B(t)$ in \eqref{B_function}, we have
\begin{equation}\label{B_definition}
B(t) = \bar{b}(t) - \underline{b}(t), \quad \bar{b}(t) = \frac{1}{\sum_{n=0}^{N-1} c^n} \sum_{i=1}^N c^{i-1} b_{i}(t), \quad \underline{b}(t) = \frac{1}{\sum_{n=0}^{N-1} c^n} \sum_{i=1}^{N} c^{N-i} b_{i}(t).
\end{equation}
It follows from \eqref{B_definition} that
\begin{equation}\label{R-1}
\begin{aligned}
m\dot{B}(t) + B(t) &= m (\dot{\bar{b}}(t) - \dot{\underline{b}}(t)) + (\bar{b}(t) - \underline{b}(t))= (m\dot{\bar{b}}(t) + \bar{b}(t)) - (m\dot{\underline{b}}(t) + \underline{b}(t)).
\end{aligned}
\end{equation}
We first estimate the term $m\dot{\bar{b}}(t) + \bar{b}(t)$ in \eqref{R-1}. It yields from \eqref{B_definition} and \eqref{f_b} that
\begin{equation*}
\begin{aligned}
&m\dot{\bar{b}}(t) + \bar{b}(t) \\
%&= m  \frac{1}{\sum_{n=0}^{N-1} c^n} \sum_{i=1}^N c^{i-1} \dot{b}_{i}(t) +  \frac{1}{\sum_{n=0}^{N-1} c^n} \sum_{i=1}^N c^{i-1} b_{i}(t)= \frac{1}{\sum_{n=0}^{N-1} c^n} \sum_{i=1}^N c^{i-1} (m\dot{b}_{i}(t) + b_{i}(t))\\
%&= \frac{1}{\sum_{n=0}^{N-1} c^n} \sum_{i=1}^N c^{i-1} \left[\frac{K}{N} \sum_{j=1}^N \psi_{ij} \cos (\theta_j(t) - \theta_i(t) + \alpha) (a_j(t) - a_i(t)) \right.\\
%&\qquad \qquad \qquad\qquad\qquad \left.- \frac{K}{N} \sum_{j=1}^N \psi_{ij}\sin (\theta_j(t) - \theta_i(t) + \alpha) (\omega_j(t) - \omega_i(t))^2 \right]\\
& = \frac{K}{N(\sum_{n=0}^{N-1} c^n)} \sum_{i=1}^N c^{i-1}  \sum_{j=1}^N \psi_{ij} \cos (\theta_j(t) - \theta_i(t) + \alpha) (a_j(t) - a_i(t))\\
&- \frac{K}{N(\sum_{n=0}^{N-1} c^n)}\sum_{i=1}^N c^{i-1}  \sum_{j=1}^N \psi_{ij}[\sin (\theta_j(t) - \theta_i(t)) \cos \alpha + \cos (\theta_j(t) - \theta_i(t)) \sin \alpha] (\omega_j(t) - \omega_i(t))^2.
\end{aligned}
\end{equation*} 
We further make rough estimate to obtain
\begin{equation}\label{R-2}
\begin{aligned}
m\dot{\bar{b}}(t) + \bar{b}(t) &\le K \psi_u D_a(t) + K \psi_u[\cos \alpha D_\theta(t) +  \sin \alpha] D^2_\omega(t) \\
&\le K \psi_u D_a(t) + K \psi_u[D^\infty\cos \alpha  +  \sin \alpha] D^2_\omega(t),
\end{aligned}
\end{equation}
where we used Lemma \ref{small} and the following estimates
\begin{equation*}
|a_j(t) - a_i(t)| \le D_a(t), \quad |\sin (\theta_j(t) - \theta_i(t))| \le |\theta_j(t) - \theta_i(t)| \le D_\theta(t) ,\quad |\omega_j(t) - \omega_i(t)| \le D_\omega(t).
\end{equation*}
Next, we consider the term $m\dot{\underline{b}}(t) + \underline{b}(t)$ in \eqref{R-1}. It yields from \eqref{B_definition} and \eqref{f_b} that
\begin{equation*}
\begin{aligned}
&m\dot{\underline{b}}(t) + \underline{b}(t) \\
%&= m \frac{1}{\sum_{n=0}^{N-1} c^n} \sum_{i=1}^{N} c^{N-i} \dot{b}_{i}(t) +\frac{1}{\sum_{n=0}^{N-1} c^n} \sum_{i=1}^{N} c^{N-i} b_{i}(t)=\frac{1}{\sum_{n=0}^{N-1} c^n} \sum_{i=1}^{N} c^{N-i} [m\dot{b}_{i}(t) + b_{i}(t)]\\
%&=\frac{1}{\sum_{n=0}^{N-1} c^n} \sum_{i=1}^{N} c^{N-i} \left[\frac{K}{N} \sum_{j=1}^N \psi_{ij} \cos (\theta_j(t) - \theta_i(t) + \alpha) (a_j(t) - a_i(t)) \right.\\
%&\qquad \qquad \qquad\qquad\qquad \left.- \frac{K}{N} \sum_{j=1}^N \psi_{ij}\sin (\theta_j(t) - \theta_i(t) + \alpha) (\omega_j(t) - \omega_i(t))^2 \right]\\
&= \frac{K}{N(\sum_{n=0}^{N-1} c^n)} \sum_{i=1}^{N} c^{N-i} \sum_{j=1}^N \psi_{ij} \cos (\theta_j(t) - \theta_i(t) + \alpha) (a_j(t) - a_i(t))\\
&\quad - \frac{K}{N(\sum_{n=0}^{N-1} c^n)} \sum_{i=1}^{N} c^{N-i} \sum_{j=1}^N \psi_{ij}[\sin (\theta_j(t) - \theta_i(t)) \cos \alpha + \cos (\theta_j(t) - \theta_i(t)) \sin \alpha](\omega_j(t) - \omega_i(t))^2.
\end{aligned}
\end{equation*}
Similar to \eqref{R-2}, we have 
\begin{equation}\label{R-3}
\begin{aligned}
m\dot{\underline{b}}(t) + \underline{b}(t) &\ge -K \psi_u D_a(t) - K\psi_u [ \cos \alpha D_\theta(t) + \sin \alpha]D^2_\omega(t)\\
&\ge-K \psi_u D_a(t) - K\psi_u [ D^\infty\cos \alpha + \sin \alpha]D^2_\omega(t).
\end{aligned}
\end{equation}
Thus, we collect \eqref{R-1}, \eqref{R-2} and \eqref{R-3} to conclude
\begin{equation}\label{R-4}
\begin{aligned}
m\dot{B}(t) + B(t) 
%&= (m\dot{\bar{b}}(t) + \bar{b}(t)) - (m\dot{\underline{b}}(t) + \underline{b}(t))\\
&\le 2K \psi_u D_a(t) + 2K\psi_u [ D^\infty\cos \alpha + \sin \alpha]D^2_\omega(t)\\
&\le \frac{2K \psi_u}{\eta}A(t) + \frac{2K\psi_u [ D^\infty\cos \alpha + \sin \alpha]}{\eta^2} F^2(t), \quad t \in J_l.
\end{aligned}
\end{equation}
which completes the proof.
\end{proof}

Finally, we present the exponential decay of $\mathcal{E}_2(t)$ defined in \eqref{energy_fre} which eventually leads to the emergence of frequency synchronization.
\begin{lemma}\label{syn}
Let $(\theta(t),\omega(t))$ be a solution to \eqref{s_phs} and suppose Assumption $(\mathcal{H})_1$-$(\mathcal{H}_3)$ hold. Then, we have
\begin{equation*}
\frac{d}{dt} \mathcal{E}_2(t) \le - \frac{K \psi_l}{2N(\sum_{n=0}^{N-1} c^n)} \mathcal{E}_2(t), \quad \text{a.e.} \  t \ge t_*.
\end{equation*}
Moreover, we have
\begin{equation*}
D_\omega(t) \le \frac{\mathcal{E}_2(t_*)}{\eta} e^{- \frac{K \psi_l}{2N(\sum_{n=0}^{N-1} c^n)}(t-t_*)}, \quad t \ge t_*.
\end{equation*}
\end{lemma}
\begin{proof}
According to Lemma \ref{s_fre_dynamic} and Lemma \ref{f_B_dynamic}, and applying the similar argument as in Lemma \ref{f_A_dynamic}, we get for $t \in J_l$ given in \eqref{split-time2} with $l=1,2,\cdots$,
\begin{align}
&m\ddot{F}(t) + \dot{F}(t) \le - \frac{2K \psi_l}{N(\sum_{n=0}^{N-1} c^n)} F(t),\label{R-5}\\
&m \dot{A} + A(t) \le 2K \psi_u \frac{1}{\eta} F(t),\label{R-6}\\
&m \dot{B}(t) + B(t) \le \frac{2K \psi_u}{\eta} A(t) + \frac{2K\psi_u(D^\infty\cos \alpha + \sin \alpha)}{\eta^2}  F^2(t).\label{R-7} 
\end{align}
%\begin{equation}\label{R-5}
%m\ddot{F}(t) + \dot{F}(t) \le - \frac{2K \psi_l}{N(\sum_{n=0}^{N-1} c^n)} F(t),
%\end{equation}
%\begin{equation}\label{R-6}
%m \dot{A} + A(t) \le 2K \psi_u \frac{1}{\eta} F(t),
%\end{equation}
%\begin{equation}\label{R-7}
%m \dot{B}(t) + B(t) \le \frac{2K \psi_u}{\eta} A(t) + \frac{2K\psi_u(D^\infty\cos \alpha + \sin \alpha)}{\eta^2}  F^2(t). 
%\end{equation}
Multiplying \eqref{R-6} by $\frac{\eta \psi_l}{2N\psi_u(\sum_{n=0}^{N-1}c^n)}$, we have
\begin{equation}\label{R-8}
\begin{aligned}
\frac{\eta \psi_l}{2N\psi_u(\sum_{n=0}^{N-1}c^n)} m \dot{A}(t) + \frac{\eta \psi_l}{2N\psi_u(\sum_{n=0}^{N-1}c^n)} A(t) \le  \frac{K \psi_l}{N(\sum_{n=0}^{N-1} c^n)} F(t).
\end{aligned}
\end{equation}
Moreover, we multiply \eqref{R-7} by $2m$ to obtain
\begin{equation}\label{R-9}
\begin{aligned}
2m^2 \dot{B}(t) + 2mB(t) \le \frac{4mK \psi_u}{\eta} A(t) + \frac{4mK\psi_u(D^\infty\cos \alpha + \sin \alpha)}{\eta^2}  F^2(t).
\end{aligned}
\end{equation}
Then, we add \eqref{R-5}, \eqref{R-8} and \eqref{R-9} together to have
\begin{equation*}
\begin{aligned}
&m\ddot{F}(t) + \dot{F}(t) + \frac{\eta \psi_l}{2N\psi_u(\sum_{n=0}^{N-1}c^n)} m \dot{A}(t) + \frac{\eta \psi_l}{2N\psi_u(\sum_{n=0}^{N-1}c^n)} A(t) + 2m^2 \dot{B}(t) + 2mB(t)\\
&\le - \frac{2K \psi_l}{N(\sum_{n=0}^{N-1} c^n)} F(t) +  \frac{K \psi_l}{N(\sum_{n=0}^{N-1} c^n)} F(t) + \frac{4mK \psi_u}{\eta} A(t) + \frac{4mK\psi_u(D^\infty\cos \alpha + \sin \alpha)}{\eta^2}  F^2(t).
\end{aligned}
\end{equation*}
This further yields that
\begin{equation}\label{R-10}
\begin{aligned}
&\frac{d}{dt} \left( F(t) + \frac{\eta \psi_l}{2N\psi_u(\sum_{n=0}^{N-1}c^n)} m A(t) + 2m^2 B(t) \right) \\
&\le -m\ddot{F}(t) - \frac{\eta \psi_l}{2N\psi_u(\sum_{n=0}^{N-1}c^n)} A(t) - 2mB(t)\\
&\quad- \frac{K \psi_l}{N(\sum_{n=0}^{N-1} c^n)} F(t)+ \frac{4mK \psi_u}{\eta} A(t) + \frac{4mK\psi_u(D^\infty\cos \alpha + \sin \alpha)}{\eta^2}  F^2(t)\\
&\le - \left[ \frac{K \psi_l}{N(\sum_{n=0}^{N-1} c^n)} - \frac{4mK\psi_u(D^\infty\cos \alpha + \sin \alpha)}{\eta^2} F(t) \right]F(t)\\
&\quad - \left( \frac{\eta \psi_l}{2N\psi_u(\sum_{n=0}^{N-1}c^n)}- \frac{4mK \psi_u}{\eta}\right)A(t)   -mB(t)\\
&\le - \frac{K \psi_l}{2N(\sum_{n=0}^{N-1} c^n)} F(t) - \frac{\eta \psi_l}{4N\psi_u(\sum_{n=0}^{N-1}c^n)} A(t) -mB(t)
\end{aligned}
\end{equation}
where we used the assumptions \eqref{assume_3} and \eqref{assume_5}, \eqref{F_bound} and the following facts
\begin{equation*}
\begin{aligned}
&|\ddot{F}(t)| \le B(t), \quad F(t) \le \frac{8 (D(\Omega) + 2K \psi_u \sin \alpha)N^2\psi_u(\sum_{n=0}^{N-1} c^n)^2\gamma^2 }{mK \eta \psi _l^2 \cos \alpha \sin^2 \gamma} \quad \text{for} \ t \in J_l, \\
&\frac{4mK\psi_u(D^\infty\cos \alpha + \sin \alpha)}{\eta^2} F(t)  \le \frac{32 (D(\Omega) + 2K \psi_u \sin \alpha)(D^\infty\cos \alpha + \sin \alpha)N^2\psi^2_u(\sum_{n=0}^{N-1} c^n)^2\gamma^2 }{ \eta^3 \psi _l^2 \cos \alpha \sin^2 \gamma},    \\
%&\frac{K \psi_l}{2N(\sum_{n=0}^{N-1} c^n)} - \frac{4mK\psi_u(D^\infty\cos \alpha + \sin \alpha)}{\eta^2} F(t)  > 0,\\
&K > \frac{64 (D(\Omega) + 2K \psi_u \sin \alpha)(D^\infty\cos \alpha + \sin \alpha)N^3\psi^2_u(\sum_{n=0}^{N-1} c^n)^3\gamma^2 }{ \eta^3 \psi _l^3 \cos \alpha \sin^2 \gamma} \quad \Longrightarrow\\
&\frac{K \psi_l}{2N(\sum_{n=0}^{N-1} c^n)} - \frac{32 (D(\Omega) + 2K \psi_u \sin \alpha)(D^\infty\cos \alpha + \sin \alpha)N^2\psi^2_u(\sum_{n=0}^{N-1} c^n)^2\gamma^2 }{ \eta^3 \psi _l^2 \cos \alpha \sin^2 \gamma}> 0,\\
&mK < \frac{\eta^2 \psi_l}{16N\psi_u^2 (\sum_{n=0}^{N-1} c^n)} \quad \Longrightarrow \quad \frac{\eta \psi_l}{4N\psi_u(\sum_{n=0}^{N-1}c^n)}- \frac{4mK \psi_u}{\eta} > 0.\\
\end{aligned}
\end{equation*}
Moreover, we see from \eqref{R-10} that
\begin{equation}\label{R-11}
\begin{aligned}
&\frac{d}{dt} \left( F(t) + \frac{\eta \psi_l}{2N\psi_u(\sum_{n=0}^{N-1}c^n)} m A(t) + 2m^2 B(t) \right)\\
&\le - \frac{K \psi_l}{2N(\sum_{n=0}^{N-1} c^n)} \left( F(t) + \frac{\eta}{2K\psi_u} A(t) +  \frac{2N(\sum_{n=0}^{N-1} c^n)}{K \psi_l}mB(t)\right)\\
&\le - \frac{K \psi_l}{2N(\sum_{n=0}^{N-1} c^n)} \left( F(t) + \frac{\eta \psi_l}{2N\psi_u(\sum_{n=0}^{N-1}c^n)} m A(t) + 2m^2 B(t) \right),
\end{aligned}
\end{equation}
where we applied the assumption \eqref{assume_3} leading to
\begin{equation*}
\begin{aligned}
&mK < \frac{N(\sum_{n=0}^{N-1} c^n)}{\psi_l} \quad \Longrightarrow \quad \frac{2K\psi_u}{\eta}  \cdot \frac{\eta \psi_l}{2N\psi_u(\sum_{n=0}^{N-1}c^n)} m < 1 ,\\
&mK < \frac{N(\sum_{n=0}^{N-1} c^n)}{\psi_l} \quad \Longrightarrow \quad \frac{K \psi_l}{2N(\sum_{n=0}^{N-1} c^n)} \cdot 2m < 1. 
\end{aligned}
\end{equation*}
Thus, it follows from \eqref{energy_fre} and \eqref{R-11} that
\begin{equation*}
\frac{d}{dt} \mathcal{E}_2(t) \le - \frac{K \psi_l}{2N(\sum_{n=0}^{N-1} c^n)} \mathcal{E}_2(t), \quad t \in J_l.
\end{equation*}
As $\mathcal{E}_2(t)$ is Lipschitz continuous, we immediately have 
\begin{equation*}
\frac{d}{dt} \mathcal{E}_2(t) \le - \frac{K \psi_l}{2N(\sum_{n=0}^{N-1} c^n)} \mathcal{E}_2(t), \quad \text{a.e.} \  t \ge t_*.
\end{equation*}
This ultimately yields
\begin{equation*}
\begin{aligned}
D_\omega(t) &\le \frac{F(t)}{\eta} \le \frac{\mathcal{E}_2(t)}{\eta} \le \frac{\mathcal{E}_2(t_*)}{\eta} e^{- \frac{K \psi_l}{2N(\sum_{n=0}^{N-1} c^n)}(t-t_*)}, \quad t \ge t_*.
\end{aligned}
\end{equation*}
Therefore, we derive the desired result.
\end{proof}

Now, we are ready to prove our main result in Theorem \ref{main}.\newline

\noindent {\bf Proof of Theorem \ref{main}:} We combine Lemma \ref{small} and Lemma \ref{syn} to finish the proof of Theorem \ref{main}.

\section{Summary}\label{sec:5}
We studied the Kuramoto model under the effects of inertia and frustration on a locally coupled network, and presented sufficient conditions in terms of large coupling strength and small effects of inertia and frustration to guarantee the exponential emergence of complete frequency synchronization. Due to the lack of second-order gradient flow structure and the singularity of second-order derivative of diameter, we switch to construct new energy functions depending on the defined convex combinations, that can govern the diameters of phase and frequency. We derived first-order dissipative differential inequalities of the constructed energy functions. Based on these estimates, we showed that initial configurations distributed in a half circle will evolve to an invariant region confined in a quarter circle, and then presented that frequency diameter converges to zero exponentially fast, which implies the exponential emergence of synchronization.
%Note that we only considered a symmetric and connected interaction topology in this work, so for a more general digraph, we will make a further discussion in our furture work.

\section*{Acknowledgments}

The work of T. Zhu is supported by the National Natural Science Foundation of China (Grant No. 12201172), the Natural Science Research Project of Universities in Anhui Province, China (Project Number: 2022AH051790) and  the Talent Research Fund of Hefei University, China (Grant/Award Number: 21-22RC23). The work of X. Zhang is supported by the National Natural Science Foundation of China (Grant No. 11801194).

\section*{Appendix: Proof of Lemma \ref{phs_subadd}}
\setcounter{equation}{0}
 
\begin{proof} We only prove the first part of this Lemma since the second part can be similarly discussed.
Due to the connectivity and symmetry of the network, we can find a path between vertices $1$ and $N$:
\begin{equation}\label{path}
1 = i_0 \leftrightarrows i_1 \leftrightarrows i_2 \leftrightarrows \cdots \leftrightarrows i_{p-1} \leftrightarrows i_p = N.
\end{equation}
Note that $(i_{k-1}, i_k) \in E$ and $(i_{k}, i_{k-1}) \in E$ with $1 \le k \le p$. We see that 
\begin{equation}\label{D-0}
\sum_{i=2}^N \sum_{\substack{j=1\\ (j,i) \in E}}^{i-1} \sin (\theta_j(t) - \theta_i(t)) \le \sum_{k=1}^p \sin (\theta_{\min \{i_{k-1},i_k\}} - \theta_{\max \{i_{k-1},i_k\}}), \\
\end{equation}
%\begin{equation}\label{D-15}
%\sum_{i=1}^{N-1} \sum_{\substack{j=i+1\\(j,i) \in E}}^N \sin (\theta_j(t) - \theta_i(t)) \ge \sum_{k=0}^{p-1} \sin (\theta_{\max\{i_k,i_{k+1}\}} - \theta_{\min\{i_k,i_{k+1}\}}).
%\end{equation}
This means that it suffices to verify
\begin{equation}\label{D-1}
\sum_{k=1}^p \sin (\theta_{\min \{i_{k-1},i_k\}} - \theta_{\max \{i_{k-1},i_k\}}) \le -\sin D_\theta(t),
\end{equation}
%\begin{equation}\label{D-16}
%\sum_{k=0}^{p-1} \sin (\theta_{\max\{i_k,i_{k+1}\}} - \theta_{\min\{i_k,i_{k+1}\}}) \ge \sin D_\theta(t).
%\end{equation}
In the sequel, we split the proof of \eqref{D-1} into two cases.

\noindent $\bullet$ {\bf Case 1:} 
Consider the case that $D_\theta(t) \le \frac{\pi}{2}$. In this case, we have for all $1 \le l \le p$, 
\begin{equation}\label{D-3}
0 \le \theta_{\bar{n}_l}(t) - \theta_1(t) \le \frac{\pi}{2},
\end{equation}
where
\begin{equation}\label{D-2}
\bar{n}_l = \max_{0 \le k \le l} i_k, \quad 1 \le l \le p.
\end{equation}
We claim that
\begin{equation}\label{D-1a}
\sum_{k=1}^l \sin (\theta_{\min \{i_{k-1},i_k\}} - \theta_{\max \{i_{k-1},i_k\}}) \le \sin (\theta_1(t) - \theta_{\bar{n}_l}), \quad \text{for all} \ 1 \le l \le p.
\end{equation}
Next, we verify \eqref{D-1a} by induction principle. 

\noindent  $\star$ {\bf Step 1:} For $l= 1$ in \eqref{D-1a}, we see that
\begin{equation*}\label{D-4}
\begin{aligned}
%&\sum_{k=1}^l \sin (\theta_{\min \{i_{k-1},i_k\}} - \theta_{\max \{i_{k-1},i_k\}})\\
\sin (\theta_{\min \{i_{0},i_1\}} - \theta_{\max \{i_{0},i_1\}}) = \sin (\theta_1(t) - \theta_{i_1}(t)) = \sin (\theta_1(t) - \theta_{\bar{n}_1}(t)),
\end{aligned}
\end{equation*}
since $1 = i_0 < i_1$ and $\bar{n}_1 = \max\limits_{0 \le k \le 1} i_k = i_1$. 

\noindent $\star$ {\bf Step 2:} For $l=2$ in \eqref{D-1a}, we consider two sub-cases.

\noindent $\diamond$ If $ \bar{n}_2 = \max\{\bar{n}_1,i_2\} = \bar{n}_1$, then we have
\begin{equation}\label{D-5}
\begin{aligned}
&\sum_{k=1}^2 \sin (\theta_{\min \{i_{k-1},i_k\}} - \theta_{\max \{i_{k-1},i_k\}})\\
&= \sin (\theta_{\min \{i_{0},i_1\}} - \theta_{\max \{i_{0},i_1\}}) + \sin (\theta_{\min \{i_{1},i_2\}} - \theta_{\max \{i_{1},i_2\}})\\
&= \sin (\theta_1(t) - \theta_{i_1}(t)) + \sin (\theta_{\min \{i_{1},i_2\}} - \theta_{\max \{i_{1},i_2\}})\\
&\le \sin (\theta_1(t) - \theta_{i_1}(t)) = \sin (\theta_1(t) - \theta_{\bar{n}_1}(t))\\
&= \sin (\theta_1(t) - \theta_{\bar{n}_2}(t)).
\end{aligned}
\end{equation}

\noindent $\diamond$ If $ \bar{n}_2 = \max\{\bar{n}_1,i_2\} = i_2$, then we have
\begin{equation}\label{D-6}
\begin{aligned}
&\sum_{k=1}^2 \sin (\theta_{\min \{i_{k-1},i_k\}} - \theta_{\max \{i_{k-1},i_k\}})\\
&= \sin (\theta_{\min \{i_{0},i_1\}} - \theta_{\max \{i_{0},i_1\}}) + \sin (\theta_{\min \{i_{1},i_2\}} - \theta_{\max \{i_{1},i_2\}})\\
&= \sin (\theta_1(t) - \theta_{i_1}(t)) + \sin (\theta_{i_{1}}(t) - \theta_{i_2}(t))\\
&\le \sin (\theta_1(t) - \theta_{i_2}(t)) = \sin (\theta_1(t) - \theta_{\bar{n}_2}(t))\\
\end{aligned}
\end{equation}
where we used $i_1 < i_2$ in this subcase.

Therefore, we combine \eqref{D-5} and \eqref{D-6} to derive
\begin{equation*}\label{D-7}
\sum_{k=1}^2 \sin (\theta_{\min \{i_{k-1},i_k\}} - \theta_{\max \{i_{k-1},i_k\}}) \le \sin (\theta_1(t) - \theta_{\bar{n}_2}(t)).
\end{equation*}

\noindent $\star$ {\bf Step 3:} Assume \eqref{D-1a} holds for $l = q$, i.e.,  
\begin{equation}\label{D-8}
\sum_{k=1}^q \sin (\theta_{\min \{i_{k-1},i_k\}} - \theta_{\max \{i_{k-1},i_k\}}) \le \sin (\theta_1(t) - \theta_{\bar{n}_q}), \quad 1 \le q \le p-1,
\end{equation}
we then prove \eqref{D-1a} holds for $l = q+1$. For this,
we consider two sub-cases.

\noindent $\diamond$ If $\bar{n}_{q+1} = \max \{\bar{n}_{q}, i_{q+1}\} = \bar{n}_q$, then it follows from \eqref{D-8} that
\begin{equation}\label{D-9}
\begin{aligned}
&\sum_{k=1}^{q+1} \sin (\theta_{\min \{i_{k-1},i_k\}} - \theta_{\max \{i_{k-1},i_k\}})\\ 
&=\sum_{k=1}^q \sin (\theta_{\min \{i_{k-1},i_k\}} - \theta_{\max \{i_{k-1},i_k\}}) + \sin (\theta_{\min \{i_{q},i_{q+1}\}} - \theta_{\max \{i_{q},i_{q+1}\}})\\
&\le \sin (\theta_1(t) - \theta_{\bar{n}_q}(t)) = \sin (\theta_1(t) - \theta_{\bar{n}_{q+1}}(t)).
\end{aligned}
\end{equation}

\noindent $\diamond$ If $\bar{n}_{q+1} = \max \{\bar{n}_{q}, i_{q+1}\} = i_{q+1}$, then from \eqref{D-8}, we have
\begin{equation}\label{D-10}
\begin{aligned}
&\sum_{k=1}^{q+1} \sin (\theta_{\min \{i_{k-1},i_k\}} - \theta_{\max \{i_{k-1},i_k\}})\\ 
&=\sum_{k=1}^q \sin (\theta_{\min \{i_{k-1},i_k\}} - \theta_{\max \{i_{k-1},i_k\}}) + \sin (\theta_{\min \{i_{q},i_{q+1}\}} - \theta_{\max \{i_{q},i_{q+1}\}})\\
&\le \sin (\theta_1(t) - \theta_{\bar{n}_q}(t)) + \sin (\theta_{i_q}(t) - \theta_{i_{q+1}}(t))\le \sin(\theta_1(t) - \theta_{i_q}(t)) + \sin (\theta_{i_q}(t) - \theta_{i_{q+1}}(t)) \\
&\le \sin (\theta_1(t) - \theta_{i_{q+1}}(t)) = \sin (\theta_1(t) - \theta_{\bar{n}_{q+1}}(t)),
\end{aligned}
\end{equation}
where we used $i_q < i_{q+1}$ in this sub-case, the concave property of $\sin x, x \in [0,\pi]$ and \eqref{D-3} yielding
\begin{equation*}
0 \le \theta_{i_q}(t) - \theta_1(t) \le \theta_{\bar{n}_q} - \theta_1(t) \le \frac{\pi}{2}, \quad 0 \ge \sin(\theta_1(t) - \theta_{i_q}(t)) \ge \sin (\theta_1(t) - \theta_{\bar{n}_q}(t)).
\end{equation*}

Therefore, we combine \eqref{D-9} and \eqref{D-10} to conclude \eqref{D-1a} holds for $l = q+1$.

Then according to Step 1 - Step 3, we prove \eqref{D-1a} by induction criteria. Particularly, we take $l = p$ in \eqref{D-1a} to get
\begin{equation}\label{D-11}
\sum_{k=1}^p \sin (\theta_{\min \{i_{k-1},i_k\}} - \theta_{\max \{i_{k-1},i_k\}}) \le \sin (\theta_1(t) - \theta_{\bar{n}_p}) = \sin (\theta_1(t) - \theta_N(t)) = - \sin D_\theta(t),
\end{equation} 
since $\bar{n}_p = N$.\newline

\noindent $\bullet$ {\bf Case 2:} 
Consider the case that $\frac{\pi}{2} < D_\theta(t) < \pi$.
%there exists $1 \le l_0 \le p$ such that 
%\begin{equation*}
%\frac{\pi}{2} < \theta_{\bar{k}_{l_0}}(t) - \theta_1(t) < \gamma < \pi.
%\end{equation*}
We define 
\begin{equation*}
q = \min \{l : \frac{\pi}{2} < \theta_{\bar{n}_{l}}(t) - \theta_1(t)< \gamma < \pi , \ 1 \le l \le p\},
\end{equation*}
which is well defined since $\bar{n}_p = i_p =N$ and $D_\theta(t) = \theta_{\bar{n}_p}(t) - \theta_1(t) > \frac{\pi}{2}$. We only consider the situation $q \ge 2$ since it's more simpler for $q=1$.
This implies 
\begin{equation}\label{D-12}
0 \le \theta_{\bar{n}_l}(t) - \theta_1(t) \le \frac{\pi}{2}, \ \text{for}\ 1 \le l \le q-1 \quad \text{and} \quad  \frac{\pi}{2} < \theta_{\bar{n}_{q}}(t) - \theta_1(t) < \gamma < \pi.
\end{equation}
Then, we can apply the similar argument as in \eqref{D-1a} to get
\begin{equation}\label{D-13}
\sum_{k=1}^q \sin (\theta_{\min \{i_{k-1},i_k\}} - \theta_{\max \{i_{k-1},i_k\}}) \le \sin (\theta_1(t) - \theta_{\bar{n}_q}).
\end{equation}
From \eqref{D-12}, we see that
\begin{equation*}
\frac{\pi}{2} < \theta_{\bar{n}_{q}}(t) - \theta_1(t) \le \theta_N(t) - \theta_1(t) < \gamma < \pi, \quad \sin (\theta_1(t) - \theta_{\bar{n}_q}) \le \sin (\theta_1(t) - \theta_N(t)). 
\end{equation*}
This together with \eqref{D-13} implies that
 \begin{equation}\label{D-14}
 \begin{aligned}
\sum_{k=1}^p \sin (\theta_{\min \{i_{k-1},i_k\}} - \theta_{\max \{i_{k-1},i_k\}}) &\le \sum_{k=1}^q \sin (\theta_{\min \{i_{k-1},i_k\}} - \theta_{\max \{i_{k-1},i_k\}}) \\
&\le \sin (\theta_1(t) - \theta_N(t)) = - \sin D_\theta(t).
\end{aligned}
\end{equation}

Thus, we combine \eqref{D-11} in Case 1 and \eqref{D-14} in Case 2 to finish the verification of \eqref{D-1} which together with \eqref{D-0}  ultimately yields 
\begin{equation*}
\sum_{i=2}^N \sum_{\substack{j=1\\ (j,i) \in E}}^{i-1} \sin (\theta_j(t) - \theta_i(t)) \le -\sin D_\theta(t).
\end{equation*}
This completes the proof of the first part.
\end{proof}

\end{document}